\documentclass[onefignum,onetabnum]{siamart190516}

\usepackage{amsmath}
\usepackage{amssymb}
\usepackage{mathrsfs}
\usepackage{amsopn}
\usepackage{bm}
\usepackage{lipsum}
\usepackage{amsfonts}
\usepackage{graphicx}
\usepackage{epstopdf}
\usepackage{algorithm}
\usepackage{algpseudocode}
\usepackage{booktabs}
\ifpdf
  \DeclareGraphicsExtensions{.eps,.pdf,.png,.jpg}
\else
  \DeclareGraphicsExtensions{.eps}
\fi

\newsiamremark{example}{Example}
\newsiamremark{remark}{Remark}
\newsiamremark{hypothesis}{Hypothesis}
\crefname{hypothesis}{Hypothesis}{Hypotheses}
\newsiamthm{claim}{Claim}

\DeclareMathOperator*{\Null}{null}
\DeclareMathOperator*{\Range}{range}

\let\s=\scriptscriptstyle

\headers{Convergence analysis of inexact two-grid methods}{Xuefeng Xu and Chen-Song Zhang}

\title{{A new analytical framework for the convergence of inexact two-grid methods}\thanks{Submitted to the editors DATE.
\funding{The research of the second author was partially supported by the National Key R\&D Program of China grants 2020YFA0711900, 2020YFA0711904, and the National Science Foundation of China grant 11971472.}}}

\author{Xuefeng Xu\thanks{Corresponding author.~Department of Mathematics, Purdue University, West Lafayette, IN 47907, USA (\email{xuxuefeng@lsec.cc.ac.cn}, \email{xu1412@purdue.edu}).}
\and Chen-Song Zhang\thanks{LSEC \& NCMIS, Academy of Mathematics and Systems Science, Chinese Academy of Sciences, Beijing 100190, China, and School of Mathematical Sciences, University of Chinese Academy of Sciences, Beijing 100049, China (\email{zhangcs@lsec.cc.ac.cn}).}}

\ifpdf
\hypersetup{
  pdftitle={Convergence analysis of inexact two-grid methods},
  pdfauthor={Xuefeng Xu and Chen-Song Zhang}
}
\fi

\begin{document}

\maketitle

\begin{abstract}
Two-grid methods with exact solution of the Galerkin coarse-grid system have been well studied by the multigrid community: an elegant identity has been established to characterize the convergence factor of exact two-grid methods. In practice, however, it is often too costly to solve the Galerkin coarse-grid system exactly, especially when its size is large. Instead, without essential loss of convergence speed, one may solve the coarse-grid system approximately. In this paper, we develop a new framework for analyzing the convergence of inexact two-grid methods: two-sided bounds for the energy norm of the error propagation matrix of inexact two-grid methods are presented. In the framework, a restricted smoother involved in the identity for exact two-grid convergence is used to measure how far the actual coarse-grid matrix deviates from the Galerkin one. As an application, we establish a unified convergence theory for multigrid methods.
\end{abstract}



\section{Introduction}

Multigrid is one of the most efficient methods for solving large-scale linear systems that arise from discretized partial differential equations. It has been shown to be a powerful solver, with linear or near-linear computational complexity, for a large class of linear systems; see, e.g.,~\cite{Hackbusch1985,Briggs2000,Trottenberg2001,Vassilevski2008}. Such a desirable property stems from combining two complementary error-reduction processes: \textit{smoothing} (or \textit{relaxation}) and \textit{coarse-grid correction}. In multigrid methods, these two processes will be applied iteratively until a desired residual tolerance is achieved.

Consider solving the linear system
\begin{equation}\label{system}
A\mathbf{u}=\mathbf{f},
\end{equation}
where $A\in\mathbb{R}^{n\times n}$ is symmetric positive definite (SPD), $\mathbf{u}\in\mathbb{R}^{n}$, and $\mathbf{f}\in\mathbb{R}^{n}$. Let $M$ be an $n\times n$ nonsingular matrix such that $M+M^{T}-A$ is SPD. Given an initial guess $\mathbf{u}^{(0)}\in\mathbb{R}^{n}$, we perform the following smoothing process:
\begin{equation}\label{relax}
\mathbf{u}^{(k+1)}=\mathbf{u}^{(k)}+M^{-1}\big(\mathbf{f}-A\mathbf{u}^{(k)}\big) \quad k=0,1,\ldots
\end{equation}
The process~\cref{relax} is typically a simple iterative method, such as the (weighted) Jacobi and Gauss--Seidel iterations. In general,~\cref{relax} is efficient at eliminating high-frequency (i.e., oscillatory) error, while low-frequency (i.e., smooth) error cannot be eliminated effectively; see, e.g.,~\cite{Briggs2000,Trottenberg2001}. To further reduce the low-frequency error, a coarse-grid correction strategy is used in multigrid methods. Let $\mathbf{u}^{(\ell)}\in\mathbb{R}^{n}$ be an approximation to the exact solution $\mathbf{u}\equiv A^{-1}\mathbf{f}$, and let $P\in\mathbb{R}^{n\times n_{\rm c}}$ be a \textit{prolongation} (or \textit{interpolation}) matrix with full column rank, where $n_{\rm c} \ (<n)$ is the number of coarse variables. Then, the (exact) coarse-grid correction can be described by
\begin{equation}\label{correction}
\mathbf{u}^{(\ell+1)}=\mathbf{u}^{(\ell)}+PA_{\rm c}^{-1}P^{T}\big(\mathbf{f}-A\mathbf{u}^{(\ell)}\big),
\end{equation}
where $A_{\rm c}:=P^{T}AP\in\mathbb{R}^{n_{\rm c}\times n_{\rm c}}$ is known as the \textit{Galerkin coarse-grid matrix}. Define
\begin{equation}\label{piA}
\varPi_{A}:=PA_{\rm c}^{-1}P^{T}A.
\end{equation}
From~\cref{correction}, we have
\begin{displaymath}
\mathbf{u}-\mathbf{u}^{(\ell+1)}=(I-\varPi_{A})\big(\mathbf{u}-\mathbf{u}^{(\ell)}\big).
\end{displaymath}
Since $I-\varPi_{A}$ is a projection along (or parallel to) $\Range(P)$ onto $\Null(P^{T}A)$, it follows that
\begin{displaymath}
(I-\varPi_{A})\mathbf{e}=0 \quad\forall\,\mathbf{e}\in\Range(P).
\end{displaymath}
That is, an efficient coarse-grid correction will be achieved if the coarse space $\Range(P)$ can accurately cover the low-frequency error.

Combining~\cref{relax} and~\cref{correction} yields a two-grid procedure, which is the fundamental module of multigrid methods. A symmetric two-grid scheme (i.e., the presmoothing and postsmoothing processes are performed in a symmetric way) for solving~\cref{system} can be described by~\cref{alg:TG}. If $B_{\rm c}$ in~\cref{alg:TG} is taken to be $A_{\rm c}$, then the algorithm is called an \textit{exact} two-grid method; otherwise, it is called an \textit{inexact} two-grid method. In particular, if $\hat{\mathbf{e}}_{\rm c}=0$, then~\cref{alg:TG} contains only two smoothing steps, in which case the convergence factor is $1-\lambda_{\min}(\widetilde{M}^{-1}A)<1$, where $\widetilde{M}$ is defined by~\cref{tildM}.

\begin{algorithm}[!htbp]

\caption{\ Two-grid method.}\label{alg:TG}

\smallskip

\begin{algorithmic}[1]

\State Presmoothing: $\mathbf{u}^{(1)}\gets\mathbf{u}^{(0)}+M^{-1}\big(\mathbf{f}-A\mathbf{u}^{(0)}\big)$ \Comment{$M+M^{T}-A\in\mathbb{R}^{n\times n}$ is SPD}

\smallskip

\State Restriction: $\mathbf{r}_{\rm c}\gets P^{T}\big(\mathbf{f}-A\mathbf{u}^{(1)}\big)$ \Comment{$P\in\mathbb{R}^{n\times n_{\rm c}}$ has full column rank}

\smallskip

\State Coarse-grid correction: $\hat{\mathbf{e}}_{\rm c}\gets B_{\rm c}^{-1}\mathbf{r}_{\rm c}$ \Comment{$B_{\rm c}\in\mathbb{R}^{n_{\rm c}\times n_{\rm c}}$ is SPD}

\smallskip

\State Prolongation: $\mathbf{u}^{(2)}\gets\mathbf{u}^{(1)}+P\hat{\mathbf{e}}_{\rm c}$

\smallskip

\State Postsmoothing: $\mathbf{u}_{\rm ITG}\gets\mathbf{u}^{(2)}+M^{-T}\big(\mathbf{f}-A\mathbf{u}^{(2)}\big)$

\smallskip

\end{algorithmic}

\end{algorithm}

The convergence theory of exact two-grid methods has been well developed (see, e.g.,~\cite{Falgout2005,Zikatanov2008,Notay2015,XZ2017}), and the convergence factor of exact two-grid methods can be characterized by an elegant identity~\cite{XZ2002,Falgout2005}. In practice, however, it is often too costly to solve the linear system $A_{\rm c}\mathbf{e}_{\rm c}=\mathbf{r}_{\rm c}$ exactly, especially when its size is large. Moreover, the Galerkin coarse-grid matrix may affect the parallel efficiency of algebraic multigrid methods~\cite{Falgout2014}. Instead, without essential loss of convergence speed, one may solve the coarse-grid system approximately (see, e.g.,~\cite{Haber2007,Babich2010,Gee2011,Brannick2016}) or find a cheap alternative to $A_{\rm c}$ (see, e.g.,~\cite{Brandt2000,Sterck2006,Sterck2008,Falgout2014,Treister2015,Bienz2016}). A typical strategy is to apply~\cref{alg:TG} recursively in the correction steps. The resulting multigrid method can be viewed as an inexact two-grid scheme. This enables us to analyze the convergence of multigrid methods via inexact two-grid theory.

In~\cite{Notay2007}, an upper bound for the convergence factor of~\cref{alg:TG} was presented, which has been successfully applied to the convergence analysis of W-cycle multigrid methods. Nevertheless, the two-grid estimate in~\cite{Notay2007} is not sharp in some situations; see~\cref{rek:Notay} and~\cref{example}. Recently, we established an improved convergence theory for~\cref{alg:TG} in~\cite{XXF2022}, the basic idea of which is to measure the accuracy of the coarse solver by using the extreme eigenvalues of $B_{\rm c}^{-1}A_{\rm c}$.

In this paper, we develop a novel framework for analyzing the convergence of~\cref{alg:TG}: two-sided bounds for the energy norm of the error propagation matrix of~\cref{alg:TG} are presented, from which one can easily obtain the identity for exact two-grid convergence. In the framework, the restricted smoother $P^{T}\widetilde{M}P$ is used to measure the deviation $B_{\rm c}-A_{\rm c}$, which will be readily available if $B_{\rm c}$ is a sparsification of $A_{\rm c}$. Such an idea is completely different from that in~\cite{XXF2022}. Indeed, it is inspired by an explicit expression for the inexact two-grid preconditioner (see~\cref{lem:iTG1}). As an application of the framework, we establish a unified convergence theory for multigrid methods, which allows the coarsest-grid system to be solved approximately.

The rest of this paper is organized as follows. In~\cref{sec:pre}, we introduce some results on the convergence of two-grid methods and a useful tool for eigenvalue analysis. In~\cref{sec:iTG}, we present an analytical framework for the convergence of~\cref{alg:TG}. Based on the proposed framework, we establish a unified convergence theory for multigrid methods in~\cref{sec:MG}. In~\cref{sec:con}, we give some concluding remarks.

\section{Preliminaries} \label{sec:pre}

In this section, we review some results on the convergence of two-grid methods and the well-known Weyl's theorem. For convenience, we list some notation used in the subsequent discussions.

\begin{itemize}

\item[--] $I_{n}$ denotes the $n\times n$ identity matrix (or $I$ when its size is clear from context).

\item[--] $\lambda_{\min}(\cdot)$ and $\lambda_{\max}(\cdot)$ stand for the smallest and largest eigenvalues of a matrix, respectively.

\item[--] $\lambda(\cdot)$ denotes the spectrum of a matrix.

\item[--] $\rho(\cdot)$ denotes the spectral radius of a matrix.

\item[--] $\langle\cdot,\cdot\rangle$ denotes the standard Euclidean inner product of two vectors.

\item[--] $\|\cdot\|_{2}$ denotes the spectral norm of a matrix.

\item[--] $\|\cdot\|_{A}$ denotes the energy norm induced by an SPD matrix $A\in\mathbb{R}^{n\times n}$: for any $\mathbf{v}\in\mathbb{R}^{n}$, $\|\mathbf{v}\|_{A}=\langle A\mathbf{v},\mathbf{v}\rangle^{\frac{1}{2}}$; for any $B\in\mathbb{R}^{n\times n}$, $\|B\|_{A}=\max_{\mathbf{v}\in\mathbb{R}^{n}\backslash\{0\}}\frac{\|B\mathbf{v}\|_{A}}{\|\mathbf{v}\|_{A}}$.

\end{itemize}

From~\cref{relax}, we have
\begin{displaymath}
\mathbf{u}-\mathbf{u}^{(k+1)}=(I-M^{-1}A)\big(\mathbf{u}-\mathbf{u}^{(k)}\big),
\end{displaymath}
which leads to
\begin{displaymath}
\big\|\mathbf{u}-\mathbf{u}^{(k)}\big\|_{A}\leq\|I-M^{-1}A\|_{A}^{k}\big\|\mathbf{u}-\mathbf{u}^{(0)}\big\|_{A}.
\end{displaymath}
For any initial guess $\mathbf{u}^{(0)}\in\mathbb{R}^{n}$, if $\|I-M^{-1}A\|_{A}<1$, then
\begin{displaymath}
\lim_{k\rightarrow+\infty}\big\|\mathbf{u}-\mathbf{u}^{(k)}\big\|_{A}=0.
\end{displaymath}
Such a smoother $M$ (i.e., $\|I-M^{-1}A\|_{A}<1$) is said to be \textit{$A$-convergent}, which is in fact equivalent to the positive definiteness of $M+M^{T}-A$; see, e.g.,~\cite[Proposition~3.8]{Vassilevski2008}.

For an $A$-convergent smoother $M$, we define two \textit{symmetrized} variants:
\begin{subequations}
\begin{align}
\overline{M}&:=M(M+M^{T}-A)^{-1}M^{T},\label{barM}\\
\widetilde{M}&:=M^{T}(M+M^{T}-A)^{-1}M.\label{tildM}
\end{align}
\end{subequations}
It is easy to verify that
\begin{subequations}
\begin{align}
I-\overline{M}^{-1}A&=(I-M^{-T}A)(I-M^{-1}A),\label{rela-barM}\\
I-\widetilde{M}^{-1}A&=(I-M^{-1}A)(I-M^{-T}A),\label{rela-tildM}
\end{align}
\end{subequations}
from which one can easily deduce that both $\overline{M}-A$ and $\widetilde{M}-A$ are symmetric positive semidefinite (SPSD).

The \textit{iteration matrix} (or \textit{error propagation matrix}) of~\cref{alg:TG} is
\begin{equation}\label{tE_TG1}
E_{\rm ITG}=(I-M^{-T}A)(I-PB_{\rm c}^{-1}P^{T}A)(I-M^{-1}A),
\end{equation}
which satisfies
\begin{displaymath}
\mathbf{u}-\mathbf{u}_{\rm ITG}=E_{\rm ITG}\big(\mathbf{u}-\mathbf{u}^{(0)}\big).
\end{displaymath}
The iteration matrix $E_{\rm ITG}$ can be expressed as
\begin{equation}\label{tE_TG2}
E_{\rm ITG}=I-B_{\rm ITG}^{-1}A,
\end{equation}
where
\begin{equation}\label{inv_tB_TG}
B_{\rm ITG}^{-1}=\overline{M}^{-1}+(I-M^{-T}A)PB_{\rm c}^{-1}P^{T}(I-AM^{-1}).
\end{equation}
Since $\overline{M}$ and $B_{\rm c}$ are SPD, it follows that $B_{\rm ITG}$ is an SPD matrix, which is called the \textit{inexact two-grid preconditioner}. By~\cref{tE_TG2}, we have
\begin{equation}\label{norm_tE_TG}
\|E_{\rm ITG}\|_{A}=\rho(E_{\rm ITG})=\max\bigg\{\frac{1}{\lambda_{\min}(A^{-1}B_{\rm ITG})}-1,\,1-\frac{1}{\lambda_{\max}(A^{-1}B_{\rm ITG})}\bigg\},
\end{equation}
which is referred to as the \textit{convergence factor} of~\cref{alg:TG}.

In particular, if $B_{\rm c}=A_{\rm c}$, the corresponding iteration matrix is denoted by $E_{\rm TG}$, which takes the form
\begin{equation}\label{E_TG}
E_{\rm TG}=(I-M^{-T}A)(I-\varPi_{A})(I-M^{-1}A)=I-B_{\rm TG}^{-1}A,
\end{equation}
where $\varPi_{A}$ is defined by~\cref{piA} and
\begin{equation}\label{invB_TG}
B_{\rm TG}^{-1}=\overline{M}^{-1}+(I-M^{-T}A)PA_{\rm c}^{-1}P^{T}(I-AM^{-1}).
\end{equation}
The SPD matrix $B_{\rm TG}$ is called the \textit{exact two-grid preconditioner}.

The following theorem provides an elegant identity for $\|E_{\rm TG}\|_{A}$~\cite[Theorem~4.3]{Falgout2005}, which is a two-level version of the XZ-identity~\cite{XZ2002,Zikatanov2008}.

\begin{theorem}
Let $\widetilde{M}$ be defined by~\cref{tildM}, and define
\begin{equation}\label{piM}
\varPi_{\widetilde{M}}:=P(P^{T}\widetilde{M}P)^{-1}P^{T}\widetilde{M}.
\end{equation}
Then, the convergence factor of~\cref{alg:TG} with $B_{\rm c}=A_{\rm c}$ can be characterized as
\begin{equation}\label{XZ}
\|E_{\rm TG}\|_{A}=1-\frac{1}{K_{\rm TG}},
\end{equation}
where
\begin{equation}\label{K_TG}
K_{\rm TG}=\max_{\mathbf{v}\in\mathbb{R}^{n}\backslash\{0\}}\frac{\big\|\big(I-\varPi_{\widetilde{M}}\big)\mathbf{v}\big\|_{\widetilde{M}}^{2}}{\|\mathbf{v}\|_{A}^{2}}.
\end{equation}
\end{theorem}

\begin{remark}
It is easy to verify that $\varPi_{\widetilde{M}}^{2}=\varPi_{\widetilde{M}}$, $\Range(\varPi_{\widetilde{M}})=\Range(P)$, and~$\varPi_{\widetilde{M}}$ is self-adjoint with respect to the inner product $\langle\cdot,\cdot\rangle_{\widetilde{M}}:=\langle\widetilde{M}\cdot,\cdot\rangle$. That is, $\varPi_{\widetilde{M}}$ is an $\widetilde{M}$-orthogonal projection onto the coarse space $\Range(P)$; see, e.g.,~\cite{Falgout2005}.
\end{remark}

\begin{remark}
From~\cref{E_TG}, we deduce that $A^{\frac{1}{2}}E_{\rm TG}A^{-\frac{1}{2}}$ is a symmetric matrix with smallest eigenvalue $0$. Since
\begin{displaymath}
A^{\frac{1}{2}}E_{\rm TG}A^{-\frac{1}{2}}=I-A^{\frac{1}{2}}B_{\rm TG}^{-1}A^{\frac{1}{2}},
\end{displaymath}
we obtain that $B_{\rm TG}-A$ is SPSD and $\lambda_{\max}\big(B_{\rm TG}^{-1}A\big)=1$. Furthermore,
\begin{equation}\label{1/K}
\lambda_{\min}\big(B_{\rm TG}^{-1}A\big)=1-\lambda_{\max}(E_{\rm TG})=1-\|E_{\rm TG}\|_{A}=\frac{1}{K_{\rm TG}}.
\end{equation}
Hence,
\begin{displaymath}
K_{\rm TG}=\frac{\lambda_{\max}\big(B_{\rm TG}^{-1}A\big)}{\lambda_{\min}\big(B_{\rm TG}^{-1}A\big)},
\end{displaymath}
that is, $K_{\rm TG}$ is the corresponding condition number when~\cref{alg:TG} with $B_{\rm c}=A_{\rm c}$ is applied as a preconditioning method.
\end{remark}

The identity~\cref{XZ} is a powerful tool for analyzing two-grid methods (see, e.g.,~\cite{Falgout2005,XZ2017,Brannick2018,XXF2018}), which reflects the interplay between smoother and coarse space. The conventional strategy of designing algebraic multigrid methods is to fix a simple smoother (like the Jacobi and Gauss--Seidel types) and then optimize the choice of coarse space. Alternatively, one may fix a coarse space and then optimize the choice of smoother. It is also possible to optimize them simultaneously~\cite{XZ2017}.

Compared with the exact two-grid case, the difficulty of inexact two-grid analysis is increased by the fact that the middle term in~\cref{tE_TG1}, $I-PB_{\rm c}^{-1}P^{T}A$, is no longer a projection. Based on the idea of hierarchical basis~\cite{Bank1988} and the minimization property of Schur complements (see, e.g.,~\cite[Theorem~3.8]{Axelsson1994}), Notay~\cite{Notay2007} derived an upper bound for the convergence factor $\|E_{\rm ITG}\|_{A}$, as described in the following theorem.

\begin{theorem}
Under the assumptions of~\cref{alg:TG}, it holds that
\begin{equation}\label{Notay}
\|E_{\rm ITG}\|_{A}\leq\max\Bigg\{1-\frac{\min\big\{1,\,\lambda_{\min}(B_{\rm c}^{-1}A_{\rm c})\big\}}{K_{\rm TG}},\,\max\big\{1,\,\lambda_{\max}(B_{\rm c}^{-1}A_{\rm c})\big\}-1\Bigg\},
\end{equation}
where $K_{\rm TG}$ is given by~\cref{K_TG}.
\end{theorem}

\begin{remark}\label{rek:Notay}
If $B_{\rm c}$ in~\cref{alg:TG} is taken to be $\omega I_{n_{\rm c}}$ with $\omega>0$, then
\begin{displaymath}
E_{\rm ITG}=(I-M^{-T}A)\bigg(I-\frac{1}{\omega}PP^{T}A\bigg)(I-M^{-1}A).
\end{displaymath}
As $\omega\rightarrow+\infty$,~\cref{alg:TG} reduces to an algorithm containing only the presmoothing and postsmoothing steps, whose convergence factor is
\begin{align*}
\big\|(I-M^{-T}A)(I-M^{-1}A)\big\|_{A}&=\lambda_{\max}\big((I-M^{-T}A)(I-M^{-1}A)\big)\\
&=\lambda_{\max}\big((I-M^{-1}A)(I-M^{-T}A)\big)\\
&=1-\lambda_{\min}(\widetilde{M}^{-1}A).
\end{align*}
It is easy to see that the upper bound in~\cref{Notay} tends to $1$ (as $\omega\rightarrow+\infty$), from which one cannot determine whether the limiting algorithm is convergent.
\end{remark}

In the next section, we will establish a new convergence theory for~\cref{alg:TG}, which is based on the following Weyl's theorem; see, e.g.,~\cite[Theorem~4.3.1]{Horn2013}.

\begin{theorem}\label{thm:Weyl}
Let $H_{1}$ and $H_{2}$ be $n\times n$ Hermitian matrices. Assume that the spectra of $H_{1}$, $H_{2}$, and $H_{1}+H_{2}$ are $\{\lambda_{i}(H_{1})\}_{i=1}^{n}$, $\{\lambda_{i}(H_{2})\}_{i=1}^{n}$, and $\{\lambda_{i}(H_{1}+H_{2})\}_{i=1}^{n}$, respectively, where $\lambda_{i}(\cdot)$ denotes the $i$th smallest eigenvalue of a matrix. Then, for~each $k=1,\ldots,n$, it holds that
\begin{displaymath}
\lambda_{k-j+1}(H_{1})+\lambda_{j}(H_{2})\leq\lambda_{k}(H_{1}+H_{2})\leq\lambda_{k+\ell}(H_{1})+\lambda_{n-\ell}(H_{2})
\end{displaymath}
for all $j=1,\ldots,k$ and $\ell=0,\ldots,n-k$. In particular, one has
\begin{subequations}
\begin{align}
\lambda_{\min}(H_{1}+H_{2})&\geq\lambda_{\min}(H_{1})+\lambda_{\min}(H_{2}),\label{Weyl-min-low}\\
\lambda_{\min}(H_{1}+H_{2})&\leq\min\big\{\lambda_{\min}(H_{1})+\lambda_{\max}(H_{2}),\,\lambda_{\max}(H_{1})+\lambda_{\min}(H_{2})\big\},\label{Weyl-min-up}\\
\lambda_{\max}(H_{1}+H_{2})&\geq\max\big\{\lambda_{\max}(H_{1})+\lambda_{\min}(H_{2}),\,\lambda_{\min}(H_{1})+\lambda_{\max}(H_{2})\big\},\label{Weyl-max-low}\\
\lambda_{\max}(H_{1}+H_{2})&\leq\lambda_{\max}(H_{1})+\lambda_{\max}(H_{2}).\label{Weyl-max-up}
\end{align}
\end{subequations}
\end{theorem}

\section{Convergence of inexact two-grid methods} \label{sec:iTG}

In this section, we develop a theoretical framework for analyzing the convergence of~\cref{alg:TG}. The main idea is to measure the deviation $B_{\rm c}-A_{\rm c}$ by using the restricted smoother $P^{T}\widetilde{M}P$, which is inspired by the following expression.

\begin{lemma}\label{lem:iTG1}
The inexact two-grid preconditioner $B_{\rm ITG}$ can be expressed as
\begin{equation}\label{tB_TG}
B_{\rm ITG}=A+(I-AM^{-T})\widetilde{M}\big(I-P(P^{T}\widetilde{M}P+B_{\rm c}-A_{\rm c})^{-1}P^{T}\widetilde{M}\big)(I-M^{-1}A).
\end{equation}
\end{lemma}

\begin{proof}
In view of~\cref{barM} and~\cref{tildM}, we have
\begin{displaymath}
\overline{M}(I-M^{-T}A)=(I-AM^{-T})\widetilde{M} \quad \text{and} \quad (I-AM^{-1})\overline{M}(I-M^{-T}A)=\widetilde{M}-A.
\end{displaymath}
Using~\cref{inv_tB_TG} and the Sherman--Morrison--Woodbury formula~\cite{Sherman1950,Woodbury1950,XXF2017}, we obtain
\begin{equation}\label{tB_TG1}
B_{\rm ITG}=\overline{M}-(I-AM^{-T})\widetilde{M}P(P^{T}\widetilde{M}P+B_{\rm c}-A_{\rm c})^{-1}P^{T}\widetilde{M}(I-M^{-1}A).
\end{equation}
Note that
\begin{equation}\label{sym-rela}
\overline{M}=A+(I-AM^{-T})\widetilde{M}(I-M^{-1}A).
\end{equation}
The expression~\cref{tB_TG} follows immediately by combining~\cref{tB_TG1} and~\cref{sym-rela}.
\end{proof}

\begin{remark}
If $B_{\rm c}=A_{\rm c}$, we get from~\cref{tB_TG} that
\begin{equation}\label{B_TG}
B_{\rm TG}=A+(I-AM^{-T})\widetilde{M}(I-\varPi_{\widetilde{M}})(I-M^{-1}A),
\end{equation}
from which one can readily see that $B_{\rm TG}-A$ is SPSD.
\end{remark}

The following lemma gives some technical eigenvalue identities used in the subsequent analysis.

\begin{lemma}\label{lem:iTG2}
The extreme eigenvalues of $(A^{-1}\widetilde{M}-I)(I-\varPi_{\widetilde{M}})$ and $(A^{-1}\widetilde{M}-I)\varPi_{\widetilde{M}}$ have the following properties:
\begin{subequations}
\begin{align}
&\lambda_{\min}\big((A^{-1}\widetilde{M}-I)(I-\varPi_{\widetilde{M}})\big)=0,\label{eig1.1}\\
&\lambda_{\max}\big((A^{-1}\widetilde{M}-I)(I-\varPi_{\widetilde{M}})\big)=K_{\rm TG}-1,\label{eig1.2}\\
&\lambda_{\min}\big((A^{-1}\widetilde{M}-I)\varPi_{\widetilde{M}}\big)=0,\label{eig1.3}\\
&\lambda_{\max}\big((A^{-1}\widetilde{M}-I)\varPi_{\widetilde{M}}\big)=\lambda_{\max}(A^{-1}\widetilde{M}\varPi_{\widetilde{M}})-1.\label{eig1.4}
\end{align}
\end{subequations}
\end{lemma}

\begin{proof}
Since $A^{-1}-\widetilde{M}^{-1}$ is SPSD and $\varPi_{\widetilde{M}}$ is an $\widetilde{M}$-orthogonal projection, it holds that
\begin{displaymath}
\lambda\big((A^{-1}\widetilde{M}-I)(I-\varPi_{\widetilde{M}})\big)=\lambda\big((A^{-1}-\widetilde{M}^{-1})^{\frac{1}{2}}\widetilde{M}(I-\varPi_{\widetilde{M}})(A^{-1}-\widetilde{M}^{-1})^{\frac{1}{2}}\big)\subset[0,+\infty).
\end{displaymath}
Similarly, one has
\begin{displaymath}
\lambda\big((A^{-1}\widetilde{M}-I)\varPi_{\widetilde{M}}\big)\subset[0,+\infty) \quad \text{and} \quad \lambda(A^{-1}\widetilde{M}\varPi_{\widetilde{M}})\subset[0,+\infty).
\end{displaymath}

Due to the fact that $\varPi_{\widetilde{M}}$ is a projection matrix of rank $n_{\rm c}$, there exists a nonsingular matrix $X\in\mathbb{R}^{n\times n}$ such that
\begin{displaymath}
X^{-1}\varPi_{\widetilde{M}}X=\begin{pmatrix}
I_{n_{\rm c}} & 0 \\
0 & 0
\end{pmatrix}.
\end{displaymath}
Let
\begin{displaymath}
X^{-1}A^{-1}\widetilde{M}X=\begin{pmatrix}
\widehat{X}_{11} & \widehat{X}_{12} \\
\widehat{X}_{21} & \widehat{X}_{22}
\end{pmatrix},
\end{displaymath}
where $\widehat{X}_{ij}\in\mathbb{R}^{m_{i}\times m_{j}}$ with $m_{1}=n_{\rm c}$ and $m_{2}=n-n_{\rm c}$. Then
\begin{align*}
&X^{-1}(A^{-1}\widetilde{M}-I)(I-\varPi_{\widetilde{M}})X=\begin{pmatrix}
0 & \widehat{X}_{12} \\
0 & \widehat{X}_{22}-I_{n-n_{\rm c}}
\end{pmatrix},\\
&X^{-1}(A^{-1}\widetilde{M}-I)\varPi_{\widetilde{M}}X=\begin{pmatrix}
\widehat{X}_{11}-I_{n_{\rm c}} & 0 \\
\widehat{X}_{21} & 0
\end{pmatrix},\\
&X^{-1}A^{-1}\widetilde{M}\varPi_{\widetilde{M}}X=\begin{pmatrix}
\widehat{X}_{11} & 0 \\
\widehat{X}_{21} & 0
\end{pmatrix},
\end{align*}
from which one can obtain the identities~\cref{eig1.1}, \cref{eig1.3}, and~\cref{eig1.4}.

From~\cref{1/K}, we have
\begin{displaymath}
K_{\rm TG}=\lambda_{\max}(A^{-1}B_{\rm TG}),
\end{displaymath}
which, together with~\cref{rela-tildM} and~\cref{B_TG}, yields
\begin{align*}
K_{\rm TG}&=1+\lambda_{\max}\big(A^{-1}(I-AM^{-T})\widetilde{M}(I-\varPi_{\widetilde{M}})(I-M^{-1}A)\big)\\
&=1+\lambda_{\max}\big((I-M^{-1}A)(I-M^{-T}A)A^{-1}\widetilde{M}(I-\varPi_{\widetilde{M}})\big)\\
&=1+\lambda_{\max}\big((A^{-1}\widetilde{M}-I)(I-\varPi_{\widetilde{M}})\big),
\end{align*}
which leads to the identity~\cref{eig1.2}. This completes the proof.
\end{proof}

Define
\begin{subequations}
\begin{align}
d_{1}&:=\frac{1}{1+\lambda_{\max}\big((P^{T}\widetilde{M}P)^{-1}(B_{\rm c}-A_{\rm c})\big)},\label{d1}\\
d_{2}&:=\frac{1}{1+\lambda_{\min}\big((P^{T}\widetilde{M}P)^{-1}(B_{\rm c}-A_{\rm c})\big)}.\label{d2}
\end{align}
\end{subequations}
Recall that $B_{\rm c}$ is SPD and $P^{T}\widetilde{M}P-A_{\rm c}$ is SPSD. We then have
\begin{align*}
&0<d_{1}<\frac{1}{1-\lambda_{\min}\big((P^{T}\widetilde{M}P)^{-1}A_{\rm c}\big)}=\frac{\lambda_{\max}(A_{\rm c}^{-1}P^{T}\widetilde{M}P)}{\lambda_{\max}(A_{\rm c}^{-1}P^{T}\widetilde{M}P)-1},\\ &0<d_{2}<\frac{1}{1-\lambda_{\max}\big((P^{T}\widetilde{M}P)^{-1}A_{\rm c}\big)}=\frac{\lambda_{\min}(A_{\rm c}^{-1}P^{T}\widetilde{M}P)}{\lambda_{\min}(A_{\rm c}^{-1}P^{T}\widetilde{M}P)-1}.
\end{align*}

We are now in a position to present a new convergence theory for~\cref{alg:TG}.

\begin{theorem}\label{thm:iTG}
Let $d_{1}$ and $d_{2}$ be defined by~\cref{d1} and~\cref{d2}, respectively. Under the assumptions of~\cref{alg:TG}, $\|E_{\rm ITG}\|_{A}$ satisfies the following estimates.

{\rm (i)} If $d_{2}\leq 1$, then
\begin{equation}\label{iTGest1.1}
\mathscr{L}_{1}\leq\|E_{\rm ITG}\|_{A}\leq\mathscr{U}_{1},
\end{equation}
where
\begin{align*}
\mathscr{L}_{1}&=1-\frac{1}{\max\big\{K_{\rm TG},\,\lambda_{\max}(A^{-1}\widetilde{M})-d_{2}\lambda_{\max}(A^{-1}\widetilde{M}\varPi_{\widetilde{M}})+d_{2}\big\}},\\
\mathscr{U}_{1}&=1-\frac{1}{d_{1}K_{\rm TG}+(1-d_{1})\lambda_{\max}(A^{-1}\widetilde{M})}.
\end{align*}

{\rm (ii)} If $d_{1}\leq 1<d_{2}<\frac{\lambda_{\max}\left(A^{-1}\widetilde{M}\varPi_{\widetilde{M}}\right)}{\lambda_{\max}\left(A^{-1}\widetilde{M}\varPi_{\widetilde{M}}\right)-1}$, then
\begin{equation}\label{iTGest1.2}
\mathscr{L}_{2}\leq\|E_{\rm ITG}\|_{A}\leq\max\big\{\mathscr{U}_{1},\,\mathscr{U}_{2}\big\},
\end{equation}
where
\begin{align*}
\mathscr{L}_{2}&=1-\frac{1}{\max\big\{\lambda_{\min}(A^{-1}\widetilde{M}),\,d_{2}K_{\rm TG}+(1-d_{2})\lambda_{\max}(A^{-1}\widetilde{M})\big\}},\\
\mathscr{U}_{2}&=\frac{1}{(1-d_{2})\lambda_{\max}(A^{-1}\widetilde{M}\varPi_{\widetilde{M}})+d_{2}}-1.
\end{align*}

{\rm (iii)} If $1<d_{1}\leq d_{2}<\frac{\lambda_{\max}\left(A^{-1}\widetilde{M}\varPi_{\widetilde{M}}\right)}{\lambda_{\max}\left(A^{-1}\widetilde{M}\varPi_{\widetilde{M}}\right)-1}$, then
\begin{equation}\label{iTGest1.3}
\max\big\{\mathscr{L}_{2},\,\mathscr{L}_{3}\big\}\leq\|E_{\rm ITG}\|_{A}\leq\mathscr{U}_{3},
\end{equation}
where
\begin{align*}
\mathscr{L}_{3}&=\frac{1}{\min\big\{\lambda_{\max}(A^{-1}\widetilde{M})-d_{1}\lambda_{\max}\big(A^{-1}\widetilde{M}\varPi_{\widetilde{M}}\big),\,(1-d_{1})\lambda_{\min}(A^{-1}\widetilde{M})\big\}+d_{1}}-1,\\
\mathscr{U}_{3}&=\max\Bigg\{1-\frac{1}{K_{\rm TG}},\,\frac{1}{(1-d_{2})\lambda_{\max}(A^{-1}\widetilde{M}\varPi_{\widetilde{M}})+d_{2}}-1\Bigg\}.
\end{align*}
\end{theorem}

\begin{proof}
By~\cref{tB_TG}, we have
\begin{displaymath}
A^{-1}B_{\rm ITG}=I+(I-M^{-T}A)A^{-1}\widetilde{M}\big(I-P(P^{T}\widetilde{M}P+B_{\rm c}-A_{\rm c})^{-1}P^{T}\widetilde{M}\big)(I-M^{-1}A).
\end{displaymath}
Then
\begin{align*}
\lambda(A^{-1}B_{\rm ITG})&=\lambda\big(I+(I-\widetilde{M}^{-1}A)A^{-1}\widetilde{M}\big(I-P(P^{T}\widetilde{M}P+B_{\rm c}-A_{\rm c})^{-1}P^{T}\widetilde{M}\big)\big)\\
&=\lambda\big(I+(A^{-1}\widetilde{M}-I)\big(I-P(P^{T}\widetilde{M}P+B_{\rm c}-A_{\rm c})^{-1}P^{T}\widetilde{M}\big)\big),
\end{align*}
which leads to
\begin{align*}
\lambda_{\min}(A^{-1}B_{\rm ITG})&=1+\lambda_{\min}\big((A^{-1}\widetilde{M}-I)\big(I-P(P^{T}\widetilde{M}P+B_{\rm c}-A_{\rm c})^{-1}P^{T}\widetilde{M}\big)\big),\\
\lambda_{\max}(A^{-1}B_{\rm ITG})&=1+\lambda_{\max}\big((A^{-1}\widetilde{M}-I)\big(I-P(P^{T}\widetilde{M}P+B_{\rm c}-A_{\rm c})^{-1}P^{T}\widetilde{M}\big)\big).
\end{align*}
Observe that $(A^{-1}\widetilde{M}-I)\big(I-P(P^{T}\widetilde{M}P+B_{\rm c}-A_{\rm c})^{-1}P^{T}\widetilde{M}\big)$ has the same eigenvalues as the symmetric matrix
\begin{displaymath}
(A^{-1}-\widetilde{M}^{-1})^{\frac{1}{2}}\widetilde{M}\big(I-P(P^{T}\widetilde{M}P+B_{\rm c}-A_{\rm c})^{-1}P^{T}\widetilde{M}\big)(A^{-1}-\widetilde{M}^{-1})^{\frac{1}{2}}.
\end{displaymath}
Since $B_{\rm c}-A_{\rm c}-\big(\frac{1}{d_{2}}-1\big)P^{T}\widetilde{M}P$ and $\big(\frac{1}{d_{1}}-1\big)P^{T}\widetilde{M}P-(B_{\rm c}-A_{\rm c})$ are SPSD, it follows that
\begin{subequations}
\begin{align}
&1+s_{2}\leq\lambda_{\min}(A^{-1}B_{\rm ITG})\leq 1+s_{1},\label{min-low-up}\\
&1+t_{2}\leq\lambda_{\max}(A^{-1}B_{\rm ITG})\leq 1+t_{1},\label{max-low-up}
\end{align}
\end{subequations}
where
\begin{align*}
s_{k}&=\lambda_{\min}\big((A^{-1}\widetilde{M}-I)(I-d_{k}\varPi_{\widetilde{M}})\big),\\
t_{k}&=\lambda_{\max}\big((A^{-1}\widetilde{M}-I)(I-d_{k}\varPi_{\widetilde{M}})\big).
\end{align*}

Next, we determine the upper bounds for $s_{1}$ and $t_{1}$, as well as the lower bounds for $s_{2}$ and $t_{2}$. The remainder of this proof is divided into three parts, which correspond to three cases stated in this theorem.

\textit{Case}~1: $d_{2}\leq 1$. By~\cref{Weyl-min-up}, we have
\begin{align*}
s_{1}&=\lambda_{\min}\big((A^{-1}\widetilde{M}-I)(I-d_{1}\varPi_{\widetilde{M}})\big)\\
&=\lambda_{\min}\big((A^{-1}-\widetilde{M}^{-1})^{\frac{1}{2}}\widetilde{M}(I-d_{1}\varPi_{\widetilde{M}})(A^{-1}-\widetilde{M}^{-1})^{\frac{1}{2}}\big)\\
&=\lambda_{\min}\big((A^{-1}-\widetilde{M}^{-1})^{\frac{1}{2}}\widetilde{M}\big(I-\varPi_{\widetilde{M}}+(1-d_{1})\varPi_{\widetilde{M}}\big)(A^{-1}-\widetilde{M}^{-1})^{\frac{1}{2}}\big)\\
&\leq\lambda_{\min}\big((A^{-1}-\widetilde{M}^{-1})^{\frac{1}{2}}\widetilde{M}(I-\varPi_{\widetilde{M}})(A^{-1}-\widetilde{M}^{-1})^{\frac{1}{2}}\big)\\
&\quad+(1-d_{1})\lambda_{\max}\big((A^{-1}-\widetilde{M}^{-1})^{\frac{1}{2}}\widetilde{M}\varPi_{\widetilde{M}}(A^{-1}-\widetilde{M}^{-1})^{\frac{1}{2}}\big)\\
&=\lambda_{\min}\big((A^{-1}\widetilde{M}-I)(I-\varPi_{\widetilde{M}})\big)+(1-d_{1})\lambda_{\max}\big((A^{-1}\widetilde{M}-I)\varPi_{\widetilde{M}}\big)\\
&=(1-d_{1})\big(\lambda_{\max}(A^{-1}\widetilde{M}\varPi_{\widetilde{M}})-1\big),
\end{align*}
where we have used the identities~\cref{eig1.1} and~\cref{eig1.4}. This suggests that Weyl's theorem (\cref{thm:Weyl}) can also be \textit{applied} to the nonsymmetric matrix $(A^{-1}\widetilde{M}-I)(I-d_{k}\varPi_{\widetilde{M}})$. Similarly,
\begin{align*}
s_{1}&=\lambda_{\min}\big((A^{-1}\widetilde{M}-I)(I-d_{1}\varPi_{\widetilde{M}})\big)\\
&\leq\lambda_{\min}(A^{-1}\widetilde{M}-I)-d_{1}\lambda_{\min}\big((A^{-1}\widetilde{M}-I)\varPi_{\widetilde{M}}\big)\\
&=\lambda_{\min}(A^{-1}\widetilde{M})-1,
\end{align*}
where we have used the fact~\cref{eig1.3}. Hence,
\begin{equation}\label{s1-up1}
s_{1}\leq\min\big\{(1-d_{1})\lambda_{\max}(A^{-1}\widetilde{M}\varPi_{\widetilde{M}})+d_{1},\,\lambda_{\min}(A^{-1}\widetilde{M})\big\}-1.
\end{equation}
Using~\cref{Weyl-min-low}, we obtain
\begin{align*}
s_{2}&=\lambda_{\min}\big((A^{-1}\widetilde{M}-I)\big((1-d_{2})I+d_{2}(I-\varPi_{\widetilde{M}})\big)\big)\\
&\geq(1-d_{2})\lambda_{\min}(A^{-1}\widetilde{M}-I)+d_{2}\lambda_{\min}\big((A^{-1}\widetilde{M}-I)(I-\varPi_{\widetilde{M}})\big),
\end{align*}
which, together with~\cref{eig1.1}, yields
\begin{equation}\label{s2-low1}
s_{2}\geq(1-d_{2})\lambda_{\min}(A^{-1}\widetilde{M})+d_{2}-1.
\end{equation}

By~\cref{Weyl-max-up}, we have
\begin{align*}
t_{1}&=\lambda_{\max}\big((A^{-1}\widetilde{M}-I)\big((1-d_{1})I+d_{1}(I-\varPi_{\widetilde{M}})\big)\big)\\
&\leq(1-d_{1})\lambda_{\max}(A^{-1}\widetilde{M}-I)+d_{1}\lambda_{\max}\big((A^{-1}\widetilde{M}-I)(I-\varPi_{\widetilde{M}})\big).
\end{align*}
The above inequality, combined with~\cref{eig1.2}, leads to
\begin{equation}\label{t1-up1}
t_{1}\leq d_{1}K_{\rm TG}+(1-d_{1})\lambda_{\max}(A^{-1}\widetilde{M})-1.
\end{equation}
Applying~\cref{Weyl-max-low} and~\cref{eig1.2}--\cref{eig1.4}, we get that
\begin{align*}
t_{2}&=\lambda_{\max}\big((A^{-1}\widetilde{M}-I)\big(I-\varPi_{\widetilde{M}}+(1-d_{2})\varPi_{\widetilde{M}}\big)\big)\\
&\geq\lambda_{\max}\big((A^{-1}\widetilde{M}-I)(I-\varPi_{\widetilde{M}})\big)+(1-d_{2})\lambda_{\min}\big((A^{-1}\widetilde{M}-I)\varPi_{\widetilde{M}}\big)\\
&=K_{\rm TG}-1
\end{align*}
and
\begin{align*}
t_{2}&=\lambda_{\max}\big((A^{-1}\widetilde{M}-I)(I-d_{2}\varPi_{\widetilde{M}})\big)\\
&\geq\lambda_{\max}(A^{-1}\widetilde{M}-I)-d_{2}\lambda_{\max}\big((A^{-1}\widetilde{M}-I)\varPi_{\widetilde{M}}\big)\\
&=\lambda_{\max}(A^{-1}\widetilde{M})-d_{2}\lambda_{\max}(A^{-1}\widetilde{M}\varPi_{\widetilde{M}})+d_{2}-1.
\end{align*}
Thus,
\begin{equation}\label{t2-low1}
t_{2}\geq\max\big\{K_{\rm TG},\,\lambda_{\max}(A^{-1}\widetilde{M})-d_{2}\lambda_{\max}(A^{-1}\widetilde{M}\varPi_{\widetilde{M}})+d_{2}\big\}-1.
\end{equation}

The estimate~\cref{iTGest1.1} then follows by combining~\cref{norm_tE_TG}, \cref{min-low-up}, \cref{max-low-up}, and~\cref{s1-up1}--\cref{t2-low1}.

\textit{Case}~2: $d_{1}\leq 1<d_{2}<\frac{\lambda_{\max}\left(A^{-1}\widetilde{M}\varPi_{\widetilde{M}}\right)}{\lambda_{\max}\left(A^{-1}\widetilde{M}\varPi_{\widetilde{M}}\right)-1}$. Since $d_{1}\leq 1$, the estimates~\cref{s1-up1} and~\cref{t1-up1} still hold. We next determine the lower bounds for $s_{2}$ and $t_{2}$. By~\cref{Weyl-min-low}, we have
\begin{align*}
s_{2}&=\lambda_{\min}\big((A^{-1}\widetilde{M}-I)\big(I-\varPi_{\widetilde{M}}+(1-d_{2})\varPi_{\widetilde{M}}\big)\big)\\
&\geq\lambda_{\min}\big((A^{-1}\widetilde{M}-I)(I-\varPi_{\widetilde{M}})\big)+(1-d_{2})\lambda_{\max}\big((A^{-1}\widetilde{M}-I)\varPi_{\widetilde{M}}\big),
\end{align*}
which, together with~\cref{eig1.1} and~\cref{eig1.4}, gives
\begin{equation}\label{s2-low2}
s_{2}\geq(1-d_{2})\lambda_{\max}(A^{-1}\widetilde{M}\varPi_{\widetilde{M}})+d_{2}-1.
\end{equation}
In light of~\cref{Weyl-max-low}, \cref{eig1.2}, and~\cref{eig1.3}, we have that
\begin{align*}
t_{2}&=\lambda_{\max}\big((A^{-1}\widetilde{M}-I)(I-d_{2}\varPi_{\widetilde{M}})\big)\\
&\geq\lambda_{\min}(A^{-1}\widetilde{M}-I)-d_{2}\lambda_{\min}\big((A^{-1}\widetilde{M}-I)\varPi_{\widetilde{M}}\big)\\
&=\lambda_{\min}(A^{-1}\widetilde{M})-1
\end{align*}
and
\begin{align*}
t_{2}&=\lambda_{\max}\big((A^{-1}\widetilde{M}-I)\big((1-d_{2})I+d_{2}(I-\varPi_{\widetilde{M}})\big)\big)\\
&\geq(1-d_{2})\lambda_{\max}(A^{-1}\widetilde{M}-I)+d_{2}\lambda_{\max}\big((A^{-1}\widetilde{M}-I)(I-\varPi_{\widetilde{M}})\big)\\
&=d_{2}K_{\rm TG}+(1-d_{2})\lambda_{\max}(A^{-1}\widetilde{M})-1.
\end{align*}
Then
\begin{equation}\label{t2-low2}
t_{2}\geq\max\big\{\lambda_{\min}(A^{-1}\widetilde{M}),\,d_{2}K_{\rm TG}+(1-d_{2})\lambda_{\max}(A^{-1}\widetilde{M})\big\}-1.
\end{equation}

Combining~\cref{norm_tE_TG}, \cref{min-low-up}, \cref{max-low-up}, \cref{s1-up1}, \cref{t1-up1}, \cref{s2-low2}, and~\cref{t2-low2}, we can arrive at the estimate~\cref{iTGest1.2}.

\textit{Case}~3: $1<d_{1}\leq d_{2}<\frac{\lambda_{\max}\left(A^{-1}\widetilde{M}\varPi_{\widetilde{M}}\right)}{\lambda_{\max}\left(A^{-1}\widetilde{M}\varPi_{\widetilde{M}}\right)-1}$. In such a case, the inequalities~\cref{s2-low2} and~\cref{t2-low2} are still valid. We then focus on the upper bounds for $s_{1}$ and $t_{1}$. By~\cref{Weyl-min-up}, \cref{eig1.1}, and~\cref{eig1.4}, we have that
\begin{align*}
s_{1}&=\lambda_{\min}\big((A^{-1}\widetilde{M}-I)(I-d_{1}\varPi_{\widetilde{M}})\big)\\
&\leq\lambda_{\max}(A^{-1}\widetilde{M}-I)-d_{1}\lambda_{\max}\big((A^{-1}\widetilde{M}-I)\varPi_{\widetilde{M}}\big)\\
&=\lambda_{\max}(A^{-1}\widetilde{M})-d_{1}\lambda_{\max}\big(A^{-1}\widetilde{M}\varPi_{\widetilde{M}}\big)+d_{1}-1
\end{align*}
and
\begin{align*}
s_{1}&=\lambda_{\min}\big((A^{-1}\widetilde{M}-I)\big((1-d_{1})I+d_{1}(I-\varPi_{\widetilde{M}})\big)\big)\\
&\leq(1-d_{1})\lambda_{\min}(A^{-1}\widetilde{M}-I)+d_{1}\lambda_{\min}\big((A^{-1}\widetilde{M}-I)(I-\varPi_{\widetilde{M}})\big)\\
&=(1-d_{1})\lambda_{\min}(A^{-1}\widetilde{M})+d_{1}-1.
\end{align*}
It follows that
\begin{equation}\label{s1-up2}
s_{1}\leq\min\big\{\lambda_{\max}(A^{-1}\widetilde{M})-d_{1}\lambda_{\max}\big(A^{-1}\widetilde{M}\varPi_{\widetilde{M}}\big),\,(1-d_{1})\lambda_{\min}(A^{-1}\widetilde{M})\big\}+d_{1}-1.
\end{equation}
Using~\cref{Weyl-max-up}, we obtain
\begin{align*}
t_{1}&=\lambda_{\max}\big((A^{-1}\widetilde{M}-I)\big(I-\varPi_{\widetilde{M}}+(1-d_{1})\varPi_{\widetilde{M}}\big)\big)\\
&\leq\lambda_{\max}\big((A^{-1}\widetilde{M}-I)(I-\varPi_{\widetilde{M}})\big)+(1-d_{1})\lambda_{\min}\big((A^{-1}\widetilde{M}-I)\varPi_{\widetilde{M}}\big),
\end{align*}
which, together with~\cref{eig1.2} and~\cref{eig1.3}, yields
\begin{equation}\label{t1-up2}
t_{1}\leq K_{\rm TG}-1.
\end{equation}

In view of~\cref{norm_tE_TG}, \cref{min-low-up}, \cref{max-low-up}, and~\cref{s2-low2}--\cref{t1-up2}, we conclude that the inequality~\cref{iTGest1.3} holds.
\end{proof}

\begin{remark}
If $B_{\rm c}=A_{\rm c}$, then $d_{1}=d_{2}=1$ and hence
\begin{align*}
\mathscr{L}_{1}&=1-\frac{1}{\max\big\{K_{\rm TG},\,\lambda_{\max}(A^{-1}\widetilde{M})-\lambda_{\max}(A^{-1}\widetilde{M}\varPi_{\widetilde{M}})+1\big\}},\\
\mathscr{U}_{1}&=1-\frac{1}{K_{\rm TG}}.
\end{align*}
By~\cref{Weyl-max-up}, \cref{eig1.2}, and~\cref{eig1.4}, we have
\begin{align*}
K_{\rm TG}+\lambda_{\max}(A^{-1}\widetilde{M}\varPi_{\widetilde{M}})-2&\geq\lambda_{\max}\big((A^{-1}\widetilde{M}-I)(I-\varPi_{\widetilde{M}})+(A^{-1}\widetilde{M}-I)\varPi_{\widetilde{M}}\big)\\
&=\lambda_{\max}(A^{-1}\widetilde{M})-1,
\end{align*}
that is,
\begin{displaymath}
K_{\rm TG}\geq\lambda_{\max}(A^{-1}\widetilde{M})-\lambda_{\max}(A^{-1}\widetilde{M}\varPi_{\widetilde{M}})+1,
\end{displaymath}
which leads to
\begin{displaymath}
\mathscr{L}_{1}=1-\frac{1}{K_{\rm TG}}.
\end{displaymath}
Thus, the estimate~\cref{iTGest1.1} will reduce to the identity~\cref{XZ} when $B_{\rm c}=A_{\rm c}$.
\end{remark}

\begin{remark}
If $B_{\rm c}=\omega I_{n_{\rm c}}$ with $\omega>0$, then
\begin{align*}
\lim_{\omega\rightarrow+\infty}d_{1}&=\lim_{\omega\rightarrow+\infty}\frac{1}{1+\lambda_{\max}\big((P^{T}\widetilde{M}P)^{-1}(\omega I_{n_{\rm c}}-A_{\rm c})\big)}=0,\\
\lim_{\omega\rightarrow+\infty}d_{2}&=\lim_{\omega\rightarrow+\infty}\frac{1}{1+\lambda_{\min}\big((P^{T}\widetilde{M}P)^{-1}(\omega I_{n_{\rm c}}-A_{\rm c})\big)}=0.
\end{align*}
It is easy to verify that, as $\omega\rightarrow+\infty$, both $\mathscr{L}_{1}$ and $\mathscr{U}_{1}$ tend to $1-\lambda_{\min}(\widetilde{M}^{-1}A)$, which is exactly the convergence factor of the limiting algorithm. That is, our estimate has fixed the defect of~\cref{Notay} indicated in~\cref{rek:Notay}.
\end{remark}

\begin{example}\label{example}
Let $A$ be partitioned into the two-by-two block form
\begin{equation}\label{2-by-2}
A=\begin{pmatrix}
A_{\rm ff} & A_{\rm fc} \\
A_{\rm cf} & A_{\rm cc}
\end{pmatrix},
\end{equation}
where $A_{\rm ff}\in\mathbb{R}^{n_{\rm f}\times n_{\rm f}}$, $A_{\rm fc}\in\mathbb{R}^{n_{\rm f}\times n_{\rm c}}$, $A_{\rm cf}=A_{\rm fc}^{T}$, and $A_{\rm cc}\in\mathbb{R}^{n_{\rm c}\times n_{\rm c}}$ $(n_{\rm f}+n_{\rm c}=n)$. The Cauchy--Bunyakowski--Schwarz (C.B.S.) constant associated with~\cref{2-by-2} (see, e.g.,~\cite{Eijkhout1991,Axelsson1994}) is defined as
\begin{displaymath}
\alpha:=\max_{\substack{\mathbf{v}_{\rm f}\in\mathbb{R}^{n_{\rm f}}\backslash\{0\} \\ \mathbf{v}_{\rm c}\in\mathbb{R}^{n_{\rm c}}\backslash\{0\}}}\frac{\mathbf{v}_{\rm f}^{T}A_{\rm fc}\mathbf{v}_{\rm c}}{\sqrt{\mathbf{v}_{\rm f}^{T}A_{\rm ff}\mathbf{v}_{\rm f}\cdot\mathbf{v}_{\rm c}^{T}A_{\rm cc}\mathbf{v}_{\rm c}}}=\big\|A_{\rm ff}^{-\frac{1}{2}}A_{\rm fc}A_{\rm cc}^{-\frac{1}{2}}\big\|_{2}.
\end{displaymath}
The positive definiteness of $A$ implies that $\alpha\in[0,1)$. Take
\begin{displaymath}
M=\begin{pmatrix}
A_{\rm ff} & 0 \\
0 & A_{\rm cc}
\end{pmatrix}, \quad P=\begin{pmatrix}
-A_{\rm ff}^{-1}A_{\rm fc} \\
I_{n_{\rm c}}
\end{pmatrix},
\end{displaymath}
and
\begin{displaymath}
B_{\rm c}=P_{0}^{T}AP_{0} \quad \text{with} \quad P_{0}=\begin{pmatrix}
0 \\
I_{n_{\rm c}}
\end{pmatrix}\in\mathbb{R}^{n\times n_{\rm c}}.
\end{displaymath}
Here, $M$ is an $A$-convergent smoother and $P$ is an ideal interpolation~\cite{Falgout2004,XXF2018}. Then, the iteration matrix $E_{\rm ITG}$ is of the form
\begin{displaymath}
E_{\rm ITG}=\begin{pmatrix}
\big(A_{\rm ff}^{-1}A_{\rm fc}A_{\rm cc}^{-1}A_{\rm cf}\big)^{2} & 0 \\
\ast & A_{\rm cc}^{-1}A_{\rm cf}A_{\rm ff}^{-1}A_{\rm fc}
\end{pmatrix}.
\end{displaymath}
Hence,
\begin{displaymath}
\|E_{\rm ITG}\|_{A}=\rho(E_{\rm ITG})=\big\|A_{\rm ff}^{-\frac{1}{2}}A_{\rm fc}A_{\rm cc}^{-\frac{1}{2}}\big\|_{2}^{2}=\alpha^{2}.
\end{displaymath}
It is easy to check that
\begin{displaymath}
K_{\rm TG}=\frac{1}{1-\alpha^{2}}, \quad \lambda_{\max}(A^{-1}\widetilde{M})=\frac{1}{1-\alpha^{2}}, \quad d_{1}=\frac{1}{1+\alpha^{2}}, \quad \text{and} \quad d_{2}\leq 1.
\end{displaymath}
An application of~\cref{iTGest1.1} yields
\begin{displaymath}
\|E_{\rm ITG}\|_{A}=\alpha^{2}.
\end{displaymath}
On the other hand, since
\begin{displaymath}
\lambda_{\min}(B_{\rm c}^{-1}A_{\rm c})=1-\alpha^{2} \quad \text{and} \quad \lambda_{\max}(B_{\rm c}^{-1}A_{\rm c})\leq 1,
\end{displaymath}
the upper bound in~\cref{Notay} gives $\alpha^{2}(2-\alpha^{2})$, which is strictly greater than $\alpha^{2}$ (unless $\alpha=0$). This example shows that the estimate~\cref{Notay} is not sharp. Furthermore, the relative error of the bound $\alpha^{2}(2-\alpha^{2})$ is
\begin{displaymath}
\bigg|\frac{\alpha^{2}(2-\alpha^{2})-\alpha^{2}}{\alpha^{2}}\bigg|=1-\alpha^{2},
\end{displaymath}
which is not tiny if $\alpha$ is bounded away from $1$.
\end{example}

\begin{remark}
The C.B.S. constant $\alpha$ can be expressed as
\begin{displaymath}
\alpha=\max_{\substack{\mathbf{v}_{\rm f}\in\mathbb{R}^{n_{\rm f}}\backslash\{0\} \\ \mathbf{v}_{\rm c}\in\mathbb{R}^{n_{\rm c}}\backslash\{0\}}}\frac{\mathbf{v}_{\rm f}^{T}S_{0}^{T}AP_{0}\mathbf{v}_{\rm c}}{\sqrt{\mathbf{v}_{\rm f}^{T}S_{0}^{T}AS_{0}\mathbf{v}_{\rm f}\cdot\mathbf{v}_{\rm c}^{T}P_{0}^{T}AP_{0}\mathbf{v}_{\rm c}}},
\end{displaymath}
where
\begin{displaymath}
S_{0}=\begin{pmatrix}
I_{n_{\rm f}} \\
0
\end{pmatrix}\in\mathbb{R}^{n\times n_{\rm f}} \quad \text{and} \quad P_{0}=\begin{pmatrix}
0 \\
I_{n_{\rm c}}
\end{pmatrix}\in\mathbb{R}^{n\times n_{\rm c}}.
\end{displaymath}
As a result, $\alpha$ can be viewed as the cosine of the abstract angle between $\Range(S_{0})$ and $\Range(P_{0})$ with respect to $A$-inner product, defined by $\langle\cdot,\cdot\rangle_{A}:=\langle A\cdot,\cdot\rangle$. As pointed out in~\cite{Notay2005}, for finite element matrices associated with the standard nodal basis, $\alpha$ is generally close to $1$, whereas, for hierarchical basis finite element matrices, $\alpha$ may be nicely bounded away from $1$; see, e.g.,~\cite{Margenov1994,Achchab1996,Achchab2001,Blaheta2003,Axelsson2004}.
\end{remark}

As an alternative to $(P^{T}\widetilde{M}P)^{-1}$, $P^{T}\widetilde{M}^{-1}P$ can also be used to analyze the convergence of~\cref{alg:TG}. It may be easier to compute $P^{T}\widetilde{M}^{-1}P$ in practice, because the action of $M^{-1}$ is always available. Indeed, there is a spectral equivalence relation between $(P^{T}\widetilde{M}P)^{-1}$ and $P^{T}\widetilde{M}^{-1}P$ (see, e.g.,~\cite[Lemma~5.2]{Falgout2004}), as discussed below.

Let $S$ be an $n\times(n-n_{\rm c})$ matrix, with full column rank, such that $P^{T}S=0$. This implies that $(S \ P)\in\mathbb{R}^{n\times n}$ is nonsingular. Let
\begin{displaymath}
L_{P}=\begin{pmatrix}
I_{n-n_{\rm c}} & 0 \\
-P^{T}\widetilde{M}S(S^{T}\widetilde{M}S)^{-1} & I_{n_{\rm c}}
\end{pmatrix}\begin{pmatrix}
0 \\
P^{T}P
\end{pmatrix}.
\end{displaymath}
Then
\begin{align*}
P^{T}\widetilde{M}^{-1}P&=P^{T}(S \ \ P)\big((S \ \ P)^{T}\widetilde{M}(S \ \ P)\big)^{-1}(S \ \ P)^{T}P\\
&=\big(0 \ \ P^{T}P\big)\begin{pmatrix}
S^{T}\widetilde{M}S & S^{T}\widetilde{M}P \\
P^{T}\widetilde{M}S & P^{T}\widetilde{M}P
\end{pmatrix}^{-1}\begin{pmatrix}
0 \\
P^{T}P
\end{pmatrix}\\
&=L_{P}^{T}\begin{pmatrix}
S^{T}\widetilde{M}S & 0 \\
0 & P^{T}\widetilde{M}P-P^{T}\widetilde{M}S(S^{T}\widetilde{M}S)^{-1}S^{T}\widetilde{M}P
\end{pmatrix}^{-1}L_{P}\\
&=P^{T}P\big(P^{T}\widetilde{M}P-P^{T}\widetilde{M}S(S^{T}\widetilde{M}S)^{-1}S^{T}\widetilde{M}P\big)^{-1}P^{T}P.
\end{align*}
Assume that the Cholesky factorization of $P^{T}P\in\mathbb{R}^{n_{\rm c}\times n_{\rm c}}$ is
\begin{displaymath}
P^{T}P=U_{\rm c}^{T}U_{\rm c},
\end{displaymath}
where $U_{\rm c}\in\mathbb{R}^{n_{\rm c}\times n_{\rm c}}$ is upper triangular. Let
\begin{displaymath}
P_{\sharp}=PU_{\rm c}^{-1},
\end{displaymath}
which is a \textit{normalized} prolongation, i.e., $P_{\sharp}^{T}P_{\sharp}=I_{n_{\rm c}}$. Then
\begin{displaymath}
\big(P_{\sharp}^{T}\widetilde{M}^{-1}P_{\sharp}\big)^{-1}=P_{\sharp}^{T}\widetilde{M}P_{\sharp}-P_{\sharp}^{T}\widetilde{M}S(S^{T}\widetilde{M}S)^{-1}S^{T}\widetilde{M}P_{\sharp},
\end{displaymath}
which yields
\begin{displaymath}
\big(P_{\sharp}^{T}\widetilde{M}P_{\sharp}\big)^{-1}\big(P_{\sharp}^{T}\widetilde{M}^{-1}P_{\sharp}\big)^{-1}=I_{n_{\rm c}}-\big(P_{\sharp}^{T}\widetilde{M}P_{\sharp}\big)^{-1}P_{\sharp}^{T}\widetilde{M}S(S^{T}\widetilde{M}S)^{-1}S^{T}\widetilde{M}P_{\sharp}.
\end{displaymath}
Thus,
\begin{equation}\label{eig-PMP}
\lambda\Big(\big(P_{\sharp}^{T}\widetilde{M}P_{\sharp}\big)^{-1}\big(P_{\sharp}^{T}\widetilde{M}^{-1}P_{\sharp}\big)^{-1}\Big)\subset[1-\beta^{2},\,1],
\end{equation}
where $\beta\in[0,1)$ is the C.B.S. constant associated with the matrix
\begin{displaymath}
\begin{pmatrix}
S^{T}\widetilde{M}S & S^{T}\widetilde{M}P_{\sharp} \\
P_{\sharp}^{T}\widetilde{M}S & P_{\sharp}^{T}\widetilde{M}P_{\sharp}
\end{pmatrix}.
\end{displaymath}
It follows from~\cref{eig-PMP} that
\begin{displaymath}
(1-\beta^{2})\mathbf{v}_{\rm c}^{T}P_{\sharp}^{T}\widetilde{M}^{-1}P_{\sharp}\mathbf{v}_{\rm c}\leq\mathbf{v}_{\rm c}^{T}\big(P_{\sharp}^{T}\widetilde{M}P_{\sharp}\big)^{-1}\mathbf{v}_{\rm c}\leq\mathbf{v}_{\rm c}^{T}P_{\sharp}^{T}\widetilde{M}^{-1}P_{\sharp}\mathbf{v}_{\rm c} \quad \forall\,\mathbf{v}_{\rm c}\in\mathbb{R}^{n_{\rm c}}.
\end{displaymath}
Some approaches to estimating the C.B.S. constant can be found, e.g., in~\cite{Eijkhout1991,Axelsson1994,Falgout2004,Vassilevski2008}.

\begin{remark}
We mention that the quantities $K_{\rm TG}$ and $\lambda_{\max}(A^{-1}\widetilde{M}\varPi_{\widetilde{M}})$ involved in~\cref{thm:iTG} will not change if $P$ is replaced by $P_{\sharp}$.
\end{remark}

\section{Convergence of multigrid methods} \label{sec:MG}

In practice, it is often too costly to solve the Galerkin coarse-grid system exactly when its size is relatively large. Instead, without essential loss of convergence speed, one may solve the coarse-grid system approximately. A typical strategy is to apply~\cref{alg:TG} recursively in the correction steps. The resulting multigrid algorithm can be treated as an inexact two-grid method. In this section, we establish a unified convergence theory for multigrid methods based on~\cref{thm:iTG}.

To describe the multigrid algorithm, we need the following notation and assumptions.

\begin{itemize}

\item The algorithm involves $L+1$ levels with indices $0,\ldots,L$, where $0$ corresponds to the coarsest level and $L$ to the finest level.

\item $n_{k}$ denotes the number of unknowns at level $k$ ($n=n_{L}>n_{L-1}>\cdots>n_{0}$).

\item For each $k=1,\ldots,L$, $P_{k}\in\mathbb{R}^{n_{k}\times n_{k-1}}$ denotes a prolongation matrix from level $k-1$ to level $k$, and ${\rm rank}(P_{k})=n_{k-1}$.

\item Let $A_{L}=A$. For each $k=0,\ldots,L-1$, $A_{k}:=P_{k+1}^{T}A_{k+1}P_{k+1}$ denotes the Galerkin coarse-grid matrix at level $k$.

\item Let $\hat{A}_{0}$ be an $n_{0}\times n_{0}$ matrix such that $\hat{A}_{0}-A_{0}$ is SPSD.

\item For each $k=1,\ldots,L$, $M_{k}\in\mathbb{R}^{n_{k}\times n_{k}}$ denotes a nonsingular smoother at level $k$ with $M_{k}+M_{k}^{T}-A_{k}$ being SPD.

\item $\gamma$ denotes the \textit{cycle index} involved in the coarse-grid correction steps.

\end{itemize}

With the above assumptions and an initial guess $\mathbf{u}_{k}^{(0)}\in\mathbb{R}^{n_{k}}$, the standard multigrid scheme for solving the linear system $A_{k}\mathbf{u}_{k}=\mathbf{f}_{k}$ (with $\mathbf{f}_{k}\in\mathbb{R}^{n_{k}}$) can be described by~\cref{alg:MG}. The symbol $\text{MG}^{\gamma}$ in~\cref{alg:MG} means that the multigrid scheme will be carried out $\gamma$ iterations. In particular, $\gamma=1$ corresponds to the V-cycle and $\gamma=2$ to the W-cycle.

\begin{algorithm}[!htbp]

\caption{\ Multigrid method at level $k$: $\mathbf{u}_{\rm IMG}\gets\textbf{MG}\big(k, A_{k}, \mathbf{f}_{k}, \mathbf{u}_{k}^{(0)}\big)$.} \label{alg:MG}

\smallskip

\begin{algorithmic}[1]

\State Presmoothing: $\mathbf{u}_{k}^{(1)}\gets\mathbf{u}_{k}^{(0)}+M_{k}^{-1}\big(\mathbf{f}_{k}-A_{k}\mathbf{u}_{k}^{(0)}\big)$

\smallskip

\State Restriction: $\mathbf{r}_{k-1}\gets P_{k}^{T}\big(\mathbf{f}_{k}-A_{k}\mathbf{u}_{k}^{(1)}\big)$

\smallskip

\State Coarse-grid correction: $\hat{\mathbf{e}}_{k-1}\gets\begin{cases}
\hat{A}_{0}^{-1}\mathbf{r}_{0} & \text{if $k=1$},\\
\textbf{MG}^{\gamma}\big(k-1, A_{k-1}, \mathbf{r}_{k-1}, \mathbf{0}\big) & \text{if $k>1$}.
\end{cases}$

\smallskip

\State Prolongation: $\mathbf{u}_{k}^{(2)}\gets\mathbf{u}_{k}^{(1)}+P_{k}\hat{\mathbf{e}}_{k-1}$

\smallskip

\State Postsmoothing: $\mathbf{u}_{\rm IMG}\gets\mathbf{u}_{k}^{(2)}+M_{k}^{-T}\big(\mathbf{f}_{k}-A_{k}\mathbf{u}_{k}^{(2)}\big)$

\smallskip

\end{algorithmic}

\end{algorithm}

From~\cref{alg:MG}, we have
\begin{displaymath}
\mathbf{u}_{k}-\mathbf{u}_{\rm IMG}=E_{\rm IMG}^{(k)}\big(\mathbf{u}_{k}-\mathbf{u}_{k}^{(0)}\big),
\end{displaymath}
where
\begin{equation}\label{tE-MGk}
E_{\rm IMG}^{(k)}=\big(I-M_{k}^{-T}A_{k}\big)\Big[I-P_{k}\Big(I-\big(E_{\rm IMG}^{(k-1)}\big)^{\gamma}\Big)A_{k-1}^{-1}P_{k}^{T}A_{k}\Big]\big(I-M_{k}^{-1}A_{k}\big).
\end{equation}
In particular,
\begin{displaymath}
E_{\rm IMG}^{(1)}=\big(I-M_{1}^{-T}A_{1}\big)\big(I-P_{1}\hat{A}_{0}^{-1}P_{1}^{T}A_{1}\big)\big(I-M_{1}^{-1}A_{1}\big).
\end{displaymath}
From~\cref{tE-MGk}, we deduce that
\begin{displaymath}
A_{k}^{\frac{1}{2}}E_{\rm IMG}^{(k)}A_{k}^{-\frac{1}{2}}=N_{k}^{T}\Big[I-A_{k}^{\frac{1}{2}}P_{k}A_{k-1}^{-\frac{1}{2}}\Big(I-\big(A_{k-1}^{\frac{1}{2}}E_{\rm IMG}^{(k-1)}A_{k-1}^{-\frac{1}{2}}\big)^{\gamma}\Big)A_{k-1}^{-\frac{1}{2}}P_{k}^{T}A_{k}^{\frac{1}{2}}\Big]N_{k},
\end{displaymath}
where
\begin{displaymath}
N_{k}=I-A_{k}^{\frac{1}{2}}M_{k}^{-1}A_{k}^{\frac{1}{2}}.
\end{displaymath}
By induction, one can get that $A_{k}^{\frac{1}{2}}E_{\rm IMG}^{(k)}A_{k}^{-\frac{1}{2}}$ is symmetric and
\begin{displaymath}
\lambda\big(E_{\rm IMG}^{(k)}\big)=\lambda\Big(A_{k}^{\frac{1}{2}}E_{\rm IMG}^{(k)}A_{k}^{-\frac{1}{2}}\Big)\subset[0,1) \quad \forall\,k=1,\ldots,L,
\end{displaymath}
which lead to
\begin{displaymath}
\big\|E_{\rm IMG}^{(k)}\big\|_{A_{k}}<1.
\end{displaymath}

Comparing~\cref{tE-MGk} with~\cref{tE_TG1}, we can observe that~\cref{alg:MG} is essentially an inexact two-grid method with $M=M_{k}$, $A=A_{k}$, $P=P_{k}$, and
\begin{equation}\label{Bc}
B_{\rm c}=A_{k-1}\Big(I-\big(E_{\rm IMG}^{(k-1)}\big)^{\gamma}\Big)^{-1}.
\end{equation}
It is easy to verify that $B_{\rm c}$ given by~\cref{Bc} is SPD.

Define
\begin{displaymath}
\sigma_{\rm TG}^{(k)}:=\big\|E_{\rm TG}^{(k)}\big\|_{A_{k}} \quad \text{and} \quad \sigma_{\rm IMG}^{(k)}:=\big\|E_{\rm IMG}^{(k)}\big\|_{A_{k}},
\end{displaymath}
which are the convergence factors of the exact two-grid method and inexact multigrid method at level $k$, respectively. By~\cref{Bc}, we have
\begin{align*}
B_{\rm c}-A_{\rm c}&=A_{k-1}\Big(I-\big(E_{\rm IMG}^{(k-1)}\big)^{\gamma}\Big)^{-1}-P_{k}^{T}A_{k}P_{k}\\
&=A_{k-1}^{\frac{1}{2}}\Big[I-\Big(A_{k-1}^{\frac{1}{2}}E_{\rm IMG}^{(k-1)}A_{k-1}^{-\frac{1}{2}}\Big)^{\gamma}\Big]^{-1}A_{k-1}^{\frac{1}{2}}-A_{k-1}\\
&=A_{k-1}^{\frac{1}{2}}\Big[I-\Big(A_{k-1}^{\frac{1}{2}}E_{\rm IMG}^{(k-1)}A_{k-1}^{-\frac{1}{2}}\Big)^{\gamma}\Big]^{-1}\Big(A_{k-1}^{\frac{1}{2}}E_{\rm IMG}^{(k-1)}A_{k-1}^{-\frac{1}{2}}\Big)^{\gamma}A_{k-1}^{\frac{1}{2}}.
\end{align*}
Define
\begin{displaymath}
\widetilde{M}_{k}:=M_{k}^{T}\big(M_{k}+M_{k}^{T}-A_{k}\big)^{-1}M_{k}.
\end{displaymath}
For any $\mathbf{v}_{k-1}\in\mathbb{R}^{n_{k-1}}\backslash\{0\}$, it holds that
\begin{align*}
\frac{\mathbf{v}_{k-1}^{T}(B_{\rm c}-A_{\rm c})\mathbf{v}_{k-1}}{\mathbf{v}_{k-1}^{T}P_{k}^{T}\widetilde{M}_{k}P_{k}\mathbf{v}_{k-1}}&=\frac{\mathbf{v}_{k-1}^{T}(B_{\rm c}-A_{\rm c})\mathbf{v}_{k-1}}{\mathbf{v}_{k-1}^{T}A_{k-1}\mathbf{v}_{k-1}}\cdot\frac{\mathbf{v}_{k-1}^{T}A_{k-1}\mathbf{v}_{k-1}}{\mathbf{v}_{k-1}^{T}P_{k}^{T}\widetilde{M}_{k}P_{k}\mathbf{v}_{k-1}}\\
&\in\Bigg[0,\,\lambda_{\max}\big((P_{k}^{T}\widetilde{M}_{k}P_{k})^{-1}A_{k-1}\big)\frac{\big(\sigma_{\rm IMG}^{(k-1)}\big)^{\gamma}}{1-\big(\sigma_{\rm IMG}^{(k-1)}\big)^{\gamma}}\Bigg],
\end{align*}
where we have used the fact
\begin{displaymath}
\lambda\Big(A_{k-1}^{\frac{1}{2}}E_{\rm IMG}^{(k-1)}A_{k-1}^{-\frac{1}{2}}\Big)\subset\Big[0,\,\sigma_{\rm IMG}^{(k-1)}\Big].
\end{displaymath}
Then
\begin{displaymath}
\frac{1}{1+\lambda_{\max}\big((P_{k}^{T}\widetilde{M}_{k}P_{k})^{-1}A_{k-1}\big)\frac{\big(\sigma_{\rm IMG}^{(k-1)}\big)^{\gamma}}{1-\big(\sigma_{\rm IMG}^{(k-1)}\big)^{\gamma}}}\leq d_{1}\leq d_{2}\leq 1,
\end{displaymath}
where $d_{1}$ and $d_{2}$ are defined by~\cref{d1} and~\cref{d2}, respectively. Since $\mathscr{L}_{1}$ and $\mathscr{U}_{1}$ are decreasing functions with respect to $d_{2}$ and $d_{1}$, respectively, it follows from~\cref{iTGest1.1} that
\begin{equation}\label{MG-lower}
\sigma_{\rm IMG}^{(k)}\geq 1-\frac{1}{K_{\rm TG}^{(k)}}=\sigma_{\rm TG}^{(k)}
\end{equation}
and
\begin{equation}\label{MG-upper}
\sigma_{\rm IMG}^{(k)}\leq 1-\frac{1+\lambda_{\max}\big((P_{k}^{T}\widetilde{M}_{k}P_{k})^{-1}A_{k-1}\big)\frac{\big(\sigma_{\rm IMG}^{(k-1)}\big)^{\gamma}}{1-\big(\sigma_{\rm IMG}^{(k-1)}\big)^{\gamma}}}{K_{\rm TG}^{(k)}+\lambda_{\max}(A_{k}^{-1}\widetilde{M}_{k})\lambda_{\max}\big((P_{k}^{T}\widetilde{M}_{k}P_{k})^{-1}A_{k-1}\big)\frac{\big(\sigma_{\rm IMG}^{(k-1)}\big)^{\gamma}}{1-\big(\sigma_{\rm IMG}^{(k-1)}\big)^{\gamma}}},
\end{equation}
where
\begin{equation}\label{K-TGk}
K_{\rm TG}^{(k)}=\max_{\mathbf{v}_{k}\in\mathbb{R}^{n_{k}}\backslash\{0\}}\frac{\big\|\big(I-\varPi_{\widetilde{M}_{k}}\big)\mathbf{v}_{k}\big\|_{\widetilde{M}_{k}}^{2}}{\|\mathbf{v}_{k}\|_{A_{k}}^{2}} \ \ \text{with} \ \ \varPi_{\widetilde{M}_{k}}=P_{k}(P_{k}^{T}\widetilde{M}_{k}P_{k})^{-1}P_{k}^{T}\widetilde{M}_{k}.
\end{equation}

\begin{remark}
The estimate~\cref{MG-lower} suggests that a well converged multigrid method entails that the corresponding exact two-grid method has a fast convergence speed.
\end{remark}

In what follows, we establish a convergence theory for~\cref{alg:MG} based on the estimate~\cref{MG-upper}. For brevity, we define
\begin{align}
\sigma_{\s L}&:=\max_{1\leq k\leq L}\sigma_{\rm TG}^{(k)},\label{sigL}\\
\tau_{\s L}&:=\max_{1\leq k\leq L}\lambda_{\max}\big((P_{k}^{T}\widetilde{M}_{k}P_{k})^{-1}A_{k-1}\big),\label{tauL}\\
\varepsilon_{\s L}&:=\min_{1\leq k\leq L}\lambda_{\min}(\widetilde{M}_{k}^{-1}A_{k}).\label{epsL}
\end{align}
In view of~\cref{K-TGk} and~\cref{epsL}, we have
\begin{displaymath}
K_{\rm TG}^{(k)}=\lambda_{\max}\big(A_{k}^{-1}\widetilde{M}_{k}\big(I-\varPi_{\widetilde{M}_{k}}\big)\big)\leq\lambda_{\max}(A_{k}^{-1}\widetilde{M}_{k})=\frac{1}{\lambda_{\min}(\widetilde{M}_{k}^{-1}A_{k})}\leq\frac{1}{\varepsilon_{\s L}}.
\end{displaymath}
Then
\begin{displaymath}
\sigma_{\rm TG}^{(k)}=1-\frac{1}{K_{\rm TG}^{(k)}}\leq 1-\varepsilon_{\s L} \quad \forall\,k=1,\ldots,L,
\end{displaymath}
which, together with~\cref{sigL}, yields
\begin{displaymath}
0\leq\sigma_{\s L}\leq 1-\varepsilon_{\s L}.
\end{displaymath}
Note that the extreme cases $\sigma_{\s L}=0$ and $\sigma_{\s L}=1-\varepsilon_{\s L}$ seldom occur in practice. In the subsequent analysis, we only consider the nontrivial case
\begin{equation}
0<\sigma_{\s L}<1-\varepsilon_{\s L}.
\end{equation}
In addition, we deduce from~\cref{tauL} that, for any $k=1,\ldots,L$,
\begin{displaymath}
\tau_{\s L}\geq\max_{\mathbf{v}_{k}\in{\rm range}(P_{k})\backslash\{0\}}\frac{\mathbf{v}_{k}^{T}A_{k}\mathbf{v}_{k}}{\mathbf{v}_{k}^{T}\widetilde{M}_{k}\mathbf{v}_{k}}\geq\min_{\mathbf{v}_{k}\in{\rm range}(P_{k})\backslash\{0\}}\frac{\mathbf{v}_{k}^{T}A_{k}\mathbf{v}_{k}}{\mathbf{v}_{k}^{T}\widetilde{M}_{k}\mathbf{v}_{k}}\geq\min_{\mathbf{v}_{k}\in\mathbb{R}^{n_{k}}\backslash\{0\}}\frac{\mathbf{v}_{k}^{T}A_{k}\mathbf{v}_{k}}{\mathbf{v}_{k}^{T}\widetilde{M}_{k}\mathbf{v}_{k}},
\end{displaymath}
which, combined with~\cref{epsL}, yields
\begin{equation}
0<\varepsilon_{\s L}\leq\tau_{\s L}.
\end{equation}

We first prove a technical lemma, which plays an important role in the convergence analysis of~\cref{alg:MG}.

\begin{lemma}\label{lem:MG}
Let $\sigma_{\s L}$, $\tau_{\s L}$, and $\varepsilon_{\s L}$ be defined by~\cref{sigL}, \cref{tauL}, and~\cref{epsL}, respectively. Then, there exists a strictly decreasing sequence $\{x_{\gamma}\}_{\gamma=1}^{+\infty}\subset(\sigma_{\s L},1-\varepsilon_{\s L})$ with limit $\sigma_{\s L}$ such that $x_{\gamma}$ is a root of the equation
\begin{displaymath}
\frac{\sigma_{\s L}\varepsilon_{\s L}(1-x^{\gamma})+\tau_{\s L}(1-\varepsilon_{\s L})(1-\sigma_{\s L})x^{\gamma}}{\varepsilon_{\s L}(1-x^{\gamma})+\tau_{\s L}(1-\sigma_{\s L})x^{\gamma}}-x=0 \quad (0<x<1).
\end{displaymath}
\end{lemma}

\begin{proof}
Let
\begin{displaymath}
F_{\gamma}(x)=\frac{\sigma_{\s L}\varepsilon_{\s L}(1-x^{\gamma})+\tau_{\s L}(1-\varepsilon_{\s L})(1-\sigma_{\s L})x^{\gamma}}{\varepsilon_{\s L}(1-x^{\gamma})+\tau_{\s L}(1-\sigma_{\s L})x^{\gamma}}-x.
\end{displaymath}
Obviously, $F_{\gamma}(x)$ is a continuous function in $(0,1)$. Direct computations yield
\begin{align*}
F_{\gamma}(\sigma_{\s L})&=\frac{\tau_{\s L}(1-\sigma_{\s L}-\varepsilon_{\s L})(1-\sigma_{\s L})\sigma_{\s L}^{\gamma}}{\varepsilon_{\s L}(1-\sigma_{\s L}^{\gamma})+\tau_{\s L}(1-\sigma_{\s L})\sigma_{\s L}^{\gamma}}>0,\\
F_{\gamma}(1-\varepsilon_{\s L})&=\frac{\varepsilon_{\s L}(1-\sigma_{\s L}-\varepsilon_{\s L})\big((1-\varepsilon_{\s L})^{\gamma}-1\big)}{\varepsilon_{\s L}-\varepsilon_{\s L}(1-\varepsilon_{\s L})^{\gamma}+\tau_{\s L}(1-\sigma_{\s L})(1-\varepsilon_{\s L})^{\gamma}}<0.
\end{align*}
Hence, $F_{\gamma}(x)=0$ has at least one root in $(\sigma_{\s L},1-\varepsilon_{\s L})$.

Let $x_{\gamma}\in(\sigma_{\s L},1-\varepsilon_{\s L})$ be a root of $F_{\gamma}(x)=0$. Note that $F_{\gamma}(x)+x$ is a strictly increasing function with respect to $x$. We then have
\begin{displaymath}
F_{\gamma+1}(x_{\gamma})=F_{\gamma}\Big(x_{\gamma}^{1+\frac{1}{\gamma}}\Big)+x_{\gamma}^{1+\frac{1}{\gamma}}-x_{\gamma}<F_{\gamma}(x_{\gamma})=0.
\end{displaymath}
Since $F_{\gamma+1}(\sigma_{\s L})>0$ and $F_{\gamma+1}(x_{\gamma})<0$, there exists an $x_{\gamma+1}\in(\sigma_{\s L},x_{\gamma})$ such that $F_{\gamma+1}(x_{\gamma+1})=0$. Repeating this process, one can obtain a strictly decreasing sequence $\{x_{\gamma}\}_{\gamma=1}^{+\infty}$.

Due to $F_{\gamma}(x_{\gamma})=0$, it follows that
\begin{displaymath}
x_{\gamma}=\frac{\sigma_{\s L}\varepsilon_{\s L}\big(1-x_{\gamma}^{\gamma}\big)+\tau_{\s L}(1-\varepsilon_{\s L})(1-\sigma_{\s L})x_{\gamma}^{\gamma}}{\varepsilon_{\s L}\big(1-x_{\gamma}^{\gamma}\big)+\tau_{\s L}(1-\sigma_{\s L})x_{\gamma}^{\gamma}},
\end{displaymath}
which leads to
\begin{displaymath}
\lim_{\gamma\rightarrow+\infty}x_{\gamma}=\sigma_{\s L}.
\end{displaymath}
This completes the proof.
\end{proof}

Using~\cref{iTGest1.1}, \cref{MG-upper}, and~\cref{lem:MG}, we can derive the following estimate.

\begin{theorem}\label{thm:MG}
Under the assumptions of~\cref{alg:MG} and~\cref{lem:MG}, if
\begin{equation}\label{condMG}
\lambda\big((P_{1}^{T}\widetilde{M}_{1}P_{1})^{-1}(\hat{A}_{0}-A_{0})\big)\subset\bigg[0,\,\frac{\varepsilon_{\s L}(x_{\gamma}-\sigma_{\s L})}{(1-\sigma_{\s L})(1-\varepsilon_{\s L}-x_{\gamma})}\bigg],
\end{equation}
then
\begin{equation}\label{estMG}
\sigma_{\rm IMG}^{(k)}\leq x_{\gamma} \quad \forall\,k=1,\ldots,L.
\end{equation}
\end{theorem}

\begin{proof}
By~\cref{iTGest1.1} and~\cref{condMG}, we have
\begin{displaymath}
\sigma_{\rm IMG}^{(1)}\leq 1-\frac{1+\frac{\varepsilon_{\s L}(x_{\gamma}-\sigma_{\s L})}{(1-\sigma_{\s L})(1-\varepsilon_{\s L}-x_{\gamma})}}{\frac{1}{1-\sigma_{\rm TG}^{(1)}}+\lambda_{\max}(A_{1}^{-1}\widetilde{M}_{1})\frac{\varepsilon_{\s L}(x_{\gamma}-\sigma_{\s L})}{(1-\sigma_{\s L})(1-\varepsilon_{\s L}-x_{\gamma})}}.
\end{displaymath}
From~\cref{sigL} and~\cref{epsL}, we deduce that
\begin{displaymath}
\frac{1}{1-\sigma_{\rm TG}^{(1)}}\leq\frac{1}{1-\sigma_{\s L}} \quad \text{and} \quad \lambda_{\max}(A_{1}^{-1}\widetilde{M}_{1})=\frac{1}{\lambda_{\min}(\widetilde{M}_{1}^{-1}A_{1})}\leq\frac{1}{\varepsilon_{\s L}}.
\end{displaymath}
Hence,
\begin{displaymath}
\sigma_{\rm IMG}^{(1)}\leq 1-\frac{1+\frac{\varepsilon_{\s L}(x_{\gamma}-\sigma_{\s L})}{(1-\sigma_{\s L})(1-\varepsilon_{\s L}-x_{\gamma})}}{\frac{1}{1-\sigma_{\s L}}+\frac{x_{\gamma}-\sigma_{\s L}}{(1-\sigma_{\s L})(1-\varepsilon_{\s L}-x_{\gamma})}}=x_{\gamma}.
\end{displaymath}
It is easy to check that
\begin{displaymath}
1-\frac{1+\xi}{\frac{1}{1-\sigma_{\rm TG}^{(k)}}+\lambda_{\max}(A_{k}^{-1}\widetilde{M}_{k})\xi} \quad \text{and} \quad \frac{\eta}{1-\eta}
\end{displaymath}
are increasing functions with respect to $\xi\in(0,+\infty)$ and $\eta\in(0,1)$, respectively. If $\sigma_{\rm IMG}^{(k-1)}\leq x_{\gamma}$, we get from~\cref{MG-upper} and~\cref{sigL}--\cref{epsL} that
\begin{align*}
\sigma_{\rm IMG}^{(k)}&\leq 1-\frac{1+\max\limits_{1\leq k\leq L}\lambda_{\max}\big((P_{k}^{T}\widetilde{M}_{k}P_{k})^{-1}A_{k-1}\big)\frac{x_{\gamma}^{\gamma}}{1-x_{\gamma}^{\gamma}}}{\frac{1}{1-\sigma_{\rm TG}^{(k)}}+\lambda_{\max}(A_{k}^{-1}\widetilde{M}_{k})\max\limits_{1\leq k\leq L}\lambda_{\max}\big((P_{k}^{T}\widetilde{M}_{k}P_{k})^{-1}A_{k-1}\big)\frac{x_{\gamma}^{\gamma}}{1-x_{\gamma}^{\gamma}}}\\
&\leq 1-\frac{1+\tau_{\s L}\frac{x_{\gamma}^{\gamma}}{1-x_{\gamma}^{\gamma}}}{\frac{1}{1-\max\limits_{1\leq k\leq L}\sigma_{\rm TG}^{(k)}}+\max\limits_{1\leq k\leq L}\lambda_{\max}(A_{k}^{-1}\widetilde{M}_{k})\tau_{\s L}\frac{x_{\gamma}^{\gamma}}{1-x_{\gamma}^{\gamma}}}\\
&=1-\frac{1+\tau_{\s L}\frac{x_{\gamma}^{\gamma}}{1-x_{\gamma}^{\gamma}}}{\frac{1}{1-\sigma_{\s L}}+\frac{\tau_{\s L}}{\varepsilon_{\s L}}\frac{x_{\gamma}^{\gamma}}{1-x_{\gamma}^{\gamma}}}\\
&=\frac{\sigma_{\s L}\varepsilon_{\s L}\big(1-x_{\gamma}^{\gamma}\big)+\tau_{\s L}(1-\varepsilon_{\s L})(1-\sigma_{\s L})x_{\gamma}^{\gamma}}{\varepsilon_{\s L}\big(1-x_{\gamma}^{\gamma}\big)+\tau_{\s L}(1-\sigma_{\s L})x_{\gamma}^{\gamma}}=x_{\gamma},
\end{align*}
where, in the last equality, we have used the fact $F_{\gamma}(x_{\gamma})=0$. The desired result then follows by induction.
\end{proof}

\begin{remark}
Observe that a key relation in the proof of~\cref{thm:MG} is
\begin{equation}\label{sig1}
\sigma_{\rm IMG}^{(1)}\leq x_{\gamma}.
\end{equation}
The purpose of the condition~\cref{condMG} is to validate such an inequality. In fact, we are allowed to replace~\cref{condMG} by any condition which can validate~\cref{sig1}. For example, if
\begin{displaymath}
\lambda\big(\hat{A}_{0}^{-1}A_{0}\big)\subset\bigg[\frac{1-\varepsilon_{\s L}-x_{\gamma}}{1-\varepsilon_{\s L}-\sigma_{\s L}},\,1\bigg],
\end{displaymath}
one can show that~\cref{sig1} is still valid (see~\cite[Theorem~4.4]{XXF2022}).
\end{remark}

In particular, we have the following convergence estimates for the V- and W-cycle multigrid methods.

\begin{corollary}\label{cor:MG}
Let
\begin{displaymath}
\mu_{\s L}=1+\sigma_{\s L}-\frac{\tau_{\s L}(1-\varepsilon_{\s L})(1-\sigma_{\s L})}{\varepsilon_{\s L}}.
\end{displaymath}
Under the assumptions of~\cref{thm:MG}, it holds that, for any $k=1,\ldots,L$,
\begin{equation}\label{cor-estMG}
\sigma_{\rm IMG}^{(k)}\leq\begin{cases}
x_{1} &\text{if $\gamma=1$},\\
\hat{x}_{2} &\text{if $\gamma=2$},
\end{cases}
\end{equation}
where
\begin{align}
x_{1}&=\frac{2\sigma_{\s L}}{\mu_{\s L}+\sqrt{\mu_{\s L}^{2}-4\sigma_{\s L}\big(1-\frac{\tau_{\s L}}{\varepsilon_{\s L}}(1-\sigma_{\s L})\big)}},\label{x1}\\
\hat{x}_{2}&=\begin{cases}
\frac{2\sigma_{\s L}}{1+\sqrt{1-4\sigma_{\s L}(1-\sigma_{\s L}-\varepsilon_{\s L})}} & \text{if $\tau_{\s L}(1-\sigma_{\s L})=\varepsilon_{\s L}$},\\[8pt]
\frac{\sigma_{\s L}\varepsilon_{\s L}(1-x_{1}^{2})+\tau_{\s L}(1-\varepsilon_{\s L})(1-\sigma_{\s L})x_{1}^{2}}{\varepsilon_{\s L}(1-x_{1}^{2})+\tau_{\s L}(1-\sigma_{\s L})x_{1}^{2}} & \text{otherwise}.
\end{cases}\label{hatx2}
\end{align}
\end{corollary}

\begin{proof}
Clearly, the equation
\begin{displaymath}
F_{1}(x)=\frac{\sigma_{\s L}\varepsilon_{\s L}(1-x)+\tau_{\s L}(1-\varepsilon_{\s L})(1-\sigma_{\s L})x}{\varepsilon_{\s L}(1-x)+\tau_{\s L}(1-\sigma_{\s L})x}-x=0 \quad (\sigma_{\s L}<x<1-\varepsilon_{\s L})
\end{displaymath}
has the same roots as
\begin{displaymath}
\big(\varepsilon_{\s L}-\tau_{\s L}(1-\sigma_{\s L})\big)x^{2}+\big(\tau_{\s L}(1-\varepsilon_{\s L})(1-\sigma_{\s L})-\varepsilon_{\s L}(1+\sigma_{\s L})\big)x+\sigma_{\s L}\varepsilon_{\s L}=0.
\end{displaymath}
If $\tau_{\s L}(1-\sigma_{\s L})=\varepsilon_{\s L}$, then the root of $F_{1}(x)=0$ is $x_{1}=\frac{\sigma_{\s L}}{\sigma_{\s L}+\varepsilon_{\s L}}$; otherwise, $x_{1}$ is of the form~\cref{x1}. Note that these two cases can be combined together.

Next, we consider the roots of
\begin{displaymath}
F_{2}(x)=\frac{\sigma_{\s L}\varepsilon_{\s L}(1-x^{2})+\tau_{\s L}(1-\varepsilon_{\s L})(1-\sigma_{\s L})x^{2}}{\varepsilon_{\s L}(1-x^{2})+\tau_{\s L}(1-\sigma_{\s L})x^{2}}-x=0 \quad (\sigma_{\s L}<x<x_{1}).
\end{displaymath}
Due to the fact that $F_{2}(x)+x$ is a strictly increasing function, it follows that
\begin{displaymath}
F_{2}(x)+x<F_{2}(x_{1})+x_{1}.
\end{displaymath}
Then
\begin{displaymath}
F_{2}(x)<\frac{\sigma_{\s L}\varepsilon_{\s L}(1-x_{1}^{2})+\tau_{\s L}(1-\varepsilon_{\s L})(1-\sigma_{\s L})x_{1}^{2}}{\varepsilon_{\s L}(1-x_{1}^{2})+\tau_{\s L}(1-\sigma_{\s L})x_{1}^{2}}-x,
\end{displaymath}
which yields
\begin{displaymath}
F_{2}\bigg(\frac{\sigma_{\s L}\varepsilon_{\s L}(1-x_{1}^{2})+\tau_{\s L}(1-\varepsilon_{\s L})(1-\sigma_{\s L})x_{1}^{2}}{\varepsilon_{\s L}(1-x_{1}^{2})+\tau_{\s L}(1-\sigma_{\s L})x_{1}^{2}}\bigg)<0.
\end{displaymath}
Since $F_{2}(\sigma_{\s L})>0$ and
\begin{displaymath}
\sigma_{\s L}<\frac{\sigma_{\s L}\varepsilon_{\s L}(1-x_{1}^{2})+\tau_{\s L}(1-\varepsilon_{\s L})(1-\sigma_{\s L})x_{1}^{2}}{\varepsilon_{\s L}(1-x_{1}^{2})+\tau_{\s L}(1-\sigma_{\s L})x_{1}^{2}}<F_{1}(x_{1})+x_{1}=x_{1},
\end{displaymath}
one can find a root $x_{2}$ satisfying that
\begin{displaymath}
x_{2}<\frac{\sigma_{\s L}\varepsilon_{\s L}(1-x_{1}^{2})+\tau_{\s L}(1-\varepsilon_{\s L})(1-\sigma_{\s L})x_{1}^{2}}{\varepsilon_{\s L}(1-x_{1}^{2})+\tau_{\s L}(1-\sigma_{\s L})x_{1}^{2}}.
\end{displaymath}
In particular, if $\tau_{\s L}(1-\sigma_{\s L})=\varepsilon_{\s L}$, then
\begin{displaymath}
x_{2}=\frac{2\sigma_{\s L}}{1+\sqrt{1-4\sigma_{\s L}(1-\sigma_{\s L}-\varepsilon_{\s L})}}.
\end{displaymath}
Thus, $\hat{x}_{2}$ given by~\cref{hatx2} is an upper bound for $x_{2}$. The estimate~\cref{cor-estMG} then follows from~\cref{thm:MG}.
\end{proof}

\begin{remark}
It is easy to see that $F_{2}(x)=0$ has the same roots as
\begin{displaymath}
\big(\varepsilon_{\s L}-\tau_{\s L}(1-\sigma_{\s L})\big)x^{3}+\big(\tau_{\s L}(1-\varepsilon_{\s L})(1-\sigma_{\s L})-\sigma_{\s L}\varepsilon_{\s L}\big)x^{2}-\varepsilon_{\s L}x+\sigma_{\s L}\varepsilon_{\s L}=0,
\end{displaymath}
which is a cubic equation if $\tau_{\s L}(1-\sigma_{\s L})\neq\varepsilon_{\s L}$. For the sake of brevity, we only give an upper bound for $x_{2}$ in~\cref{cor:MG}. Indeed, one can derive the precise expression of $x_{2}$ by using the well-known Cardano's formula.
\end{remark}

\section{Conclusions} \label{sec:con}

In this paper, we present a novel framework for analyzing the convergence of inexact two-grid methods, which is inspired by an explicit expression for the inexact two-grid preconditioner. Based on the analytical framework, we establish a unified convergence theory for multigrid methods, which allows the coarsest-grid system to be solved approximately. In the future, we expect to analyze other multilevel methods by using the proposed framework.

\section*{Acknowledgments}

This work is based on Xu's Ph.D.~thesis~\cite{XXF-thesis} at the Academy of Mathematics and Systems Science, Chinese Academy of Sciences. The authors would like to thank the anonymous referees for their valuable comments and suggestions, which greatly improved the original version of this paper.

\bibliographystyle{siamplain}
\bibliography{references}

\begin{thebibliography}{10}

\bibitem{Achchab2001}
{\sc B.~Achchab, O.~Axelsson, L.~Laayouni, and A.~Souissi}, {\em {Strengthened
  Cauchy--Bunyakow-ski--Schwarz inequality for a three-dimensional elasticity
  system}}, Numer. Linear Algebra Appl., 8 (2001), pp.~191--205.

\bibitem{Achchab1996}
{\sc B.~Achchab and J.~F. Ma\^{i}tre}, {\em {Estimate of the constant in two
  strengthened C.B.S. inequalities for F.E.M. systems of 2D elasticity:
  Application to multilevel methods and a posteriori error estimators}}, Numer.
  Linear Algebra Appl., 3 (1996), pp.~147--159.

\bibitem{Axelsson1994}
{\sc O.~Axelsson}, {\em {Iterative Solution Methods}}, Cambridge University
  Press, Cambridge, UK, 1994.

\bibitem{Axelsson2004}
{\sc O.~Axelsson and R.~Blaheta}, {\em {Two simple derivations of universal
  bounds for the C.B.S. inequality constant}}, Appl. Math., 49 (2004),
  pp.~57--72.

\bibitem{Babich2010}
{\sc R.~Babich, J.~Brannick, R.~C. Brower, M.~A. Clark, T.~A. Manteuffel, S.~F.
  McCormick, J.~C. Osborn, and C.~Rebbi}, {\em {Adaptive multigrid algorithm
  for the lattice Wilson--Dirac operator}}, Phys. Rev. Lett., 105 (2010),
  p.~201602.

\bibitem{Bank1988}
{\sc R.~E. Bank, T.~F. Dupont, and H.~Yserentant}, {\em {The hierarchical basis
  multigrid method}}, Numer. Math., 52 (1988), pp.~427--458.

\bibitem{Bienz2016}
{\sc A.~Bienz, R.~D. Falgout, W.~Gropp, L.~N. Olson, and J.~B. Schroder}, {\em
  {Reducing parallel communication in algebraic multigrid through
  sparsification}}, SIAM J. Sci. Comput., 38 (2016), pp.~S332--S357.

\bibitem{Blaheta2003}
{\sc R.~Blaheta}, {\em {Nested tetrahedral grids and strengthened C.B.S.
  inequality}}, Numer. Linear Algebra Appl., 10 (2003), pp.~619--637.

\bibitem{Brandt2000}
{\sc A.~E. Brandt}, {\em {General highly accurate algebraic coarsening}},
  Electron. Trans. Numer. Anal., 10 (2000), pp.~1--20.

\bibitem{Brannick2018}
{\sc J.~Brannick, F.~Cao, K.~Kahl, R.~D. Falgout, and X.~Hu}, {\em {Optimal
  interpolation and compatible relaxation in classical algebraic multigrid}},
  SIAM J. Sci. Comput., 40 (2018), pp.~A1473--A1493.

\bibitem{Brannick2016}
{\sc J.~Brannick, A.~Frommer, K.~Kahl, B.~Leder, M.~Rottmann, and A.~Strebel},
  {\em {Multigrid preconditioning for the overlap operator in lattice QCD}},
  Numer. Math., 132 (2016), pp.~463--490.

\bibitem{Briggs2000}
{\sc W.~L. Briggs, V.~E. Henson, and S.~F. McCormick}, {\em {A Multigrid
  Tutorial}}, 2nd ed., SIAM, Philadelphia, 2000.

\bibitem{Eijkhout1991}
{\sc V.~Eijkhout and P.~S. Vassilevski}, {\em {The role of the strengthened
  Cauchy--Buniakowskii--Schwarz inequality in multilevel methods}}, SIAM Rev.,
  33 (1991), pp.~405--419.

\bibitem{Falgout2014}
{\sc R.~D. Falgout and J.~B. Schroder}, {\em {Non-Galerkin coarse grids for
  algebraic multigrid}}, SIAM J. Sci. Comput., 36 (2014), pp.~C309--C334.

\bibitem{Falgout2004}
{\sc R.~D. Falgout and P.~S. Vassilevski}, {\em {On generalizing the algebraic
  multigrid framework}}, SIAM J. Numer. Anal., 42 (2004), pp.~1669--1693.

\bibitem{Falgout2005}
{\sc R.~D. Falgout, P.~S. Vassilevski, and L.~T. Zikatanov}, {\em {On two-grid
  convergence estimates}}, Numer. Linear Algebra Appl., 12 (2005),
  pp.~471--494.

\bibitem{Gee2011}
{\sc M.~W. Gee, U.~K\"{u}ttler, and W.~A. Wall}, {\em {Truly monolithic
  algebraic multigrid for fluid--structure interaction}}, Int. J. Numer. Meth.
  Engng, 85 (2011), pp.~987--1016.

\bibitem{Haber2007}
{\sc E.~Haber and S.~Heldmann}, {\em {An octree multigrid method for
  quasi-static Maxwell's equations with highly discontinuous coefficients}}, J.
  Comput. Phys., 223 (2007), pp.~783--796.

\bibitem{Hackbusch1985}
{\sc W.~Hackbusch}, {\em {Multi-Grid Methods and Applications}},
  Springer-Verlag, Berlin, 1985.

\bibitem{Horn2013}
{\sc R.~A. Horn and C.~R. Johnson}, {\em {Matrix Analysis}}, 2nd ed., Cambridge
  University Press, Cambridge, UK, 2013.

\bibitem{Margenov1994}
{\sc S.~D. Margenov}, {\em {Upper bound of the constant in the strengthened
  C.B.S. inequality for FEM 2D elasticity equations}}, Numer. Linear Algebra
  Appl., 1 (1994), pp.~65--74.

\bibitem{Notay2005}
{\sc Y.~Notay}, {\em {Algebraic multigrid and algebraic multilevel methods: A
  theoretical comparison}}, Numer. Linear Algebra Appl., 12 (2005),
  pp.~419--451.

\bibitem{Notay2007}
{\sc Y.~Notay}, {\em {Convergence analysis of perturbed two-grid and multigrid
  methods}}, SIAM J. Numer. Anal., 45 (2007), pp.~1035--1044.

\bibitem{Notay2015}
{\sc Y.~Notay}, {\em {Algebraic theory of two-grid methods}}, Numer. Math.
  Theory Methods Appl., 8 (2015), pp.~168--198.

\bibitem{Sherman1950}
{\sc J.~Sherman and W.~J. Morrison}, {\em {Adjustment of an inverse matrix
  corresponding to a change in one element of a given matrix}}, Ann. Math.
  Stat., 21 (1950), pp.~124--127.

\bibitem{Sterck2008}
{\sc H.~D. Sterck, R.~D. Falgout, J.~W. Nolting, and U.~M. Yang}, {\em
  {Distance-two interpolation for parallel algebraic multigrid}}, Numer. Linear
  Algebra Appl., 15 (2008), pp.~115--139.

\bibitem{Sterck2006}
{\sc H.~D. Sterck, U.~M. Yang, and J.~J. Heys}, {\em {Reducing complexity in
  parallel algebraic multigrid preconditioners}}, SIAM J. Matrix Anal. Appl.,
  27 (2006), pp.~1019--1039.

\bibitem{Treister2015}
{\sc E.~Treister and I.~Yavneh}, {\em {Non-Galerkin multigrid based on
  sparsified smoothed aggregation}}, SIAM J. Sci. Comput., 37 (2015),
  pp.~A30--A54.

\bibitem{Trottenberg2001}
{\sc U.~Trottenberg, C.~W. Oosterlee, and A.~Sch\"{u}ller}, {\em {Multigrid}},
  Academic Press, London, 2001.

\bibitem{Vassilevski2008}
{\sc P.~S. Vassilevski}, {\em {Multilevel Block Factorization Preconditioners:
  Matrix-Based Analysis and Algorithms for Solving Finite Element Equations}},
  Springer, New York, 2008.

\bibitem{Woodbury1950}
{\sc M.~A. Woodbury}, {\em {Inverting modified matrices}}, \rm in Memorandum
  Rept. 42, Statistical Research Group, Princeton University, Princeton, NJ,
  1950.

\bibitem{XZ2002}
{\sc J.~Xu and L.~T. Zikatanov}, {\em {The method of alternating projections
  and the method of subspace corrections in Hilbert space}}, J. Amer. Math.
  Soc., 15 (2002), pp.~573--597.

\bibitem{XZ2017}
{\sc J.~Xu and L.~T. Zikatanov}, {\em {Algebraic multigrid methods}}, Acta
  Numer., 26 (2017), pp.~591--721.

\bibitem{XXF-thesis}
{\sc X.~Xu}, {\em {Algebraic Theory of Multigrid Methods}}, Ph.D.~thesis,
  University of Chinese Academy of Sciences, 2019 (in chinese).

\bibitem{XXF2017}
{\sc X.~Xu}, {\em {Generalization of the Sherman--Morrison--Woodbury formula
  involving the Schur complement}}, Appl. Math. Comput., 309 (2017),
  pp.~183--191.

\bibitem{XXF2018}
{\sc X.~Xu and C.-S. Zhang}, {\em {On the ideal interpolation operator in
  algebraic multigrid methods}}, SIAM J. Numer. Anal., 56 (2018),
  pp.~1693--1710.

\bibitem{XXF2022}
{\sc X.~Xu and C.-S. Zhang}, {\em {Convergence analysis of inexact two-grid
  methods: A theoretical framework}}, SIAM J. Numer. Anal., 60 (2022),
  pp.~133--156.

\bibitem{Zikatanov2008}
{\sc L.~T. Zikatanov}, {\em {Two-sided bounds on the convergence rate of
  two-level methods}}, Numer. Linear Algebra Appl., 15 (2008), pp.~439--454.

\end{thebibliography}

\end{document}